\documentclass[12pt,a4paper]{article}
\usepackage{amsmath,amssymb,amsthm,bm,booktabs, xcolor}
\usepackage{pifont}
\usepackage{graphicx}
\usepackage{txfonts}

\newtheorem{theorem}{Theorem}[section]
\newtheorem{proposition}[theorem]{Proposition}

\newtheorem{lemma}[theorem]{Lemma}
\newtheorem{remark}[theorem]{Remark}
\newtheorem{example}[theorem]{Example}

\newcommand{\1}{{\bm 1}}
\numberwithin{equation}{section}

\begin{document}

\begin{center}
\large\bf
Quadratic Embedding Constants of Corona Graphs
\end{center}

\bigskip

\begin{center}
Ferdi \\
Faculty of Mathematics and Natural Sciences \\
Institut Teknologi Bandung, Bandung, Indonesia \\
ferdimath123@gmail.com
\end{center}

\begin{center}
Edy Tri Baskoro \\
Combinatorial Mathematics Research Group \\
Faculty of Mathematics and Natural Sciences \\
Institut Teknologi Bandung, Indonesia \\
and \\
Center for Research Collaboration \\
on Graph Theory and Combinatorics, Indonesia \\
ebaskoro@itb.ac.id
\end{center}

\begin{center}
Nobuaki Obata\\
Center for Data-driven Science and Artificial Intelligence \\
Tohoku University\\
Sendai 980-8576 Japan \\
obata@tohoku.ac.jp
\end{center}

\begin{center}
Aditya Purwa Santika \\
Combinatorial Mathematics Research Group \\
Faculty of Mathematics and Natural Sciences \\
Institut Teknologi Bandung, Indonesia \\
aditps@itb.ac.id
\end{center}

\bigskip

\begin{quote}
\textbf{Abstract}\enspace
The quadratic embedding constant (QEC) of a connected graph
is defined to be the maximum of the quadratic function associated with
its distance matrix on a certain unit sphere of codimension two.
In this paper we derive a formula for the QEC
of a corona graph $G\odot H$.
It is shown that $\mathrm{QEC}(G\odot H)
=\psi_{H*}^{-1}(\mathrm{QEC}(G))$ holds
under some spectral assumptions on $H$,
where $\psi_{H*}^{-1}$ is the inverse function of the most right
branch of the analytic function $\psi_H$ defined by means of
the main eigenvalues of the adjacency matrix of $H$.
Moreover, if $H$ is a regular graph of which the 
adjacency matrix has the smallest eigenvalue $-2$,
then the formula is written down explicitly.
\end{quote}

\begin{quote}
\end{quote}

\begin{quote}
\end{quote}

\begin{quote}
{\bfseries Acknowledgements} \enspace 
This work was supported by 
the International Research Collaboration Program of 
Institut Teknologi Bandung, the
PMDSU Program funded by the Ministry of Higher Education, 
Science and Technology, Indonesia,
and the JSPS Grant-in-Aid for Scientific Research 23K03126, Japan,
\end{quote}

\section{Introduction}
\label{01sec:Introduction}

Let $G=(V,E)$ be a finite connected graph with $|V|=n\ge2$
and $D=[d(x,y)]_{x,y\in V}$ the distance matrix of $G$,
where $d(x,y)$ stands for the 
standard graph distance between two vertices $x,y\in V$.
A \textit{quadratic embedding} of $G$ 
into a Euclidean space $\mathcal{H}$ is
a map $\varphi:V\to\mathcal{H}$ such that
\[
\|\varphi(x)-\varphi(y)\|^2=d(x,y), \quad x,y\in V.
\]
In general, a given graph does not
necessarily admit a quadratic embedding;
when it does, it is said to be \textit{of QE class}.
In the terminology of Euclidean distance geometry,
a graph is of QE class if and only if the distance matrix $D$
is a Euclidean distance matrix
\cite{Alfakih2018,Balaji-Bapat2007,Jaklic-Modic2013,
Jaklic-Modic2014, Liberti-Lavor-Maculan-Mucherino2014}.
It is known from the celebrated Schoenberg's theorem 
\cite{Schoenberg1935,Schoenberg1937,Schoenberg1938a,Schoenberg1938b}
that a connected graph $G$ is of QE class
if and only if its distance matrix $D$ is
conditionally negative definite.
To identify this condition, it is natural to consider the 
\textit{quadratic embedding constant (QEC)} defined by
\[
\mathrm{QEC}(G)
=\max\{\langle f,Df \rangle\,;\, f\in C(V), \,
\langle f,f \rangle=1, \, \langle \1,f \rangle=0\},
\]
where $C(V)$ is the space of real column vectors $f=[f_x]_{x\in V}$
with the canonical inner product $\langle\cdot,\cdot\rangle$,
and $\bm{1}\in C(V)$ is the column vector whose entries are all one.
In fact, a connected graph $G$ is of QE class
if and only if $\mathrm{QEC}(G)\le 0$.

Since the QEC was first introduced 
by Obata and Zakiyyah \cite{Obata-Zakiyyah2018},
interest has grown in characterization and
classification of graphs using the QEC
as a new numeric invariant of connected graphs.
The QECs of graphs of some particular classes are
determined explicitly, including
complete graphs and cycles \cite{Obata-Zakiyyah2018},
paths \cite{Mlotkowski2022},
complete multipartite graphs \cite{Obata2023b},
fan graphs \cite{Mlotkowski-Obata2025a, Mlotkowski-Obata2025b},
wheels \cite{Obata2017},
squid and kite graphs \cite{Purwaningsih-Sugeng2021},
hairy cycle graphs \cite{Irawan-Sugeng2021},
double star graphs \cite{Choudhury-Nandi2023}.
It is also interesting to obtain explicit formulas 
and estimates of QEC along with various graph operations,
for example, the Cartesian product \cite{Obata-Zakiyyah2018},
star product \cite{Baskoro-Obata2021, 
Mlotkowski-Obata2020, Obata-Zakiyyah2018}, 
lexicographic product \cite{Lou-Obata-Huang2022}, 
graph join \cite{Lou-Obata-Huang2022, Mlotkowski-Obata2025a},
and comb graphs (cluster graphs) \cite{Choudhury-Nandi2023}. 
Moreover, an effort was initiated to characterize
particular block graphs \cite{Baskoro-Obata2024}
and to construct primary non-QE graphs \cite{Obata2023a,Obata2023b}.

In this paper, 
along the line of research
on QEC in relation to various graph operations,
we focus on the corona of two graphs $G$ and $H$,
denoted by $G\odot H$,
for definition see Section \ref{02subsec:Corona graph}.
Since the corona graph was first discussed by
Frucht and Harary \cite{Frucht-Harary1970}
in relation to automorphism groups,
it has been actively studied as a benchmark class of graphs
along with chemical graph theory, graph coloring, 
graph labeling, network models, and so on.
For relevant spectral results we refer to
\cite{Barik-Pati-Sarma2007} for the adjacency spectra
and \cite{Indulal-Stevanovic2015} for the distance spectra.

The main purpose of this paper is
to derive a formula for $\mathrm{QEC}(G\odot H)$,
where $G$ is a connected graph on two or more vertices,
$H$ an arbitrary graph and 
$G\odot H$ their corona graph.
To describe $\mathrm{QEC}(G\odot H)$ we need
a special analytic function associated with $H$.
Let $A_H$ be the adjacency matrix of $H$,
$\mathrm{ev}(A_H)$ the set of distinct eigenvalues of $A_H$,
and $\mathrm{ev}_M(A_H)$ the subset of main ones.
Let $\omega_H(\lambda)$ be the unique
real analytic function with maximal domain satisfying
\begin{equation}\label{01eqn:def of omega(lambda) for H}
\omega_H(\lambda)
= 1+\lambda \langle\mathbf{1},(A_H+2+\lambda)^{-1}\mathbf{1}\rangle.
\end{equation}
It is shown that $\omega_H(\lambda)$ has a simple pole
at $\lambda=-\alpha-2$ for each
$\alpha\in \mathrm{ev}_{\mathrm{M}}(A_H)\backslash\{-2\}$ and,
is analytic except those poles.
Moreover, $\omega_H(\lambda)$ has 
$k_H=|\mathrm{ev}_{\mathrm{M}}(A_H)\backslash\{-2\}|\ge1$
distinct real zeroes.
We then introduce
a real analytic function $\psi_H(\lambda)$ with 
the maximal domain uniquely specified by
\begin{equation}\label{01eqn:def of psi(lambda) for H}
\psi_H(\lambda)=\frac{\lambda}{\omega_H(\lambda)}\,.
\end{equation}
It is shown that $\psi_H(\lambda)$ has a simple pole
at each zero of $\omega_H(\lambda)$,
and is analytic except those poles.
Let $\lambda_*$ be the largest zero of $\omega_H(\lambda)$.
Then $\psi_{H*}=\psi_H\!\restriction_{(\lambda_*,+\infty)}$
becomes an increasing function varying from $-\infty$ to $+\infty$.
Thus, we may define the inverse function 
$\lambda=\psi_{H*}^{-1}(\lambda^\prime)$
for $\lambda^\prime\in\mathbb{R}$,
namely, $\lambda$ is the largest
solution to the equation $\psi_H(\lambda)=\lambda^\prime$.
With those notations we will prove that
the formula
\begin{equation}\label{01eqn:main formula}
\mathrm{QEC}(G\odot H)
=\psi_{H*}^{-1}(\mathrm{QEC}(G))
\end{equation}
holds under some spectral assumptions,
for the precise statements
see Sections \ref{04subsec:Formula for QEC of corona graph}
and \ref{04subsec:Formula for QEC of corona graph with regular H}.
It is also interesting to note that if $H$ is a regular graph
with $\min\mathrm{ev}(A_H)=-2$,
formula \eqref{01eqn:main formula} holds for $\mathrm{QEC}(G)>0$,
and $\mathrm{QEC}(G\odot H)=0$
for any $G$ with $\mathrm{QEC}(G)\le0$.
We thus expect further profound relation to the graphs $H$ with
$\min\mathrm{ev}(A_H)=-2$, which have been extensively studied so far,
see e.g., the celebrated books 
\cite{Brouwer-Cohen-Neumaier1989,
Cvetkovic-Rowlinson-Simic2004,Cvetkovic-Rowlinson-Simic2010}.

The paper is organized as follows.
In Section 2 we obtain a useful expression of the distance matrix
of a corona graph and derive the set of basic equations
from which $\mathrm{QEC}(G\odot H)$ is determined.
In Section 3 we develop a general theory on
two key functions in
\eqref{01eqn:def of omega(lambda) for H} and
\eqref{01eqn:def of psi(lambda) for H}.
In Section 4 we obtain an intermediate
expression for $\mathrm{QEC}(G\odot H)$ 
as the greater or equal of $\psi_{H*}^{-1}(\mathrm{QEC}(G))$
and a certain constant $\gamma_3$ determined by
$\mathrm{ev}(A_H)$,
see Proposition \ref{04prop:preliminary formula with gamma3};
and then specify several spectral conditions
under which formula \eqref{01eqn:main formula} holds,
see Theorems \ref{04thm:min formula (01)}--\ref{04thm:min formula (03)}.
Moreover, if $H$ is a regular graph
the formula becomes simpler,
see Theorem \ref{04thm:main formula for regular graph}.
For a regular graph $H$ with $\min\mathrm{ev}(A_H)=-2$,
that formula \eqref{01eqn:main formula} holds for $\mathrm{QEC}(G)>0$
and $\mathrm{QEC}(G\odot H)=0$ otherwise,
see Theorem \ref{04thm:QEC for H with minev=-2}.
In Section 5 we present some examples.
In particular, we show that the main results
on double star graphs and particular cluster graphs
(comb graphs) in \cite{Choudhury-Nandi2023}
are covered by our main formula.

\bigskip
{\bfseries Notations:}\enspace
For any non-empty finite set $X$ 
the space of real column vectors $f=[f_x]_{x\in X}$ is
denoted by $C(X)$.
For each $x\in X$ let $e_x\in C(X)$ be the canonical unit vector
and $\bm{1}\in C(V)$ the vector of which the entries are all one,
that is $\bm{1}=\sum_{x\in X} e_x$.
Let $Y$ be another non-empty finite set.
The space of real matrices
$k=[k_{x,y}]_{x\in X,y\in Y}$ is denoted by $C(X\times Y)$.
The all-one matrix is denoted by $J_{X,Y}$ or simply by $J$.
If $X=Y$, the identity matrix is denoted by $I_X$ or simply by $I$. 
Recall that there is a canonical isomorphism
$C(X\times Y)\cong C(X)\otimes C(Y)$.

\section{Preliminaries}
\label{01sec:Preliminaries}

\subsection{Corona graph}
\label{02subsec:Corona graph}

Let $G=(V_1,E_1)$ and $H=(V_2,E_2)$ be 
arbitrary graphs with $V_1\cap V_2=\emptyset$.
For each vertex of $G$ we associate a copy of $H$.
A graph obtained from $G$ and $|V_1|$ copies of $H$
by adding edges connecting
each vertex of $G$ and every vertex of the corresponding copy of $H$
is called the \textit{corona} of $G$ and $H$,
and is denoted by $G\odot H$.
Here we note that, as a binary operation of graphs,
$G\odot H$ is neither commutative nor associative.
Obviously, the corona graph $G\odot H$ is connected
whenever so is $G$ (even if $H$ is disconnected).

By definition the vertex set of $G\odot H$ 
is the union of $V_1$ and $|V_1|$ copies of $V_2$.
To describe it, letting $o$ be a new symbol (vertex)
which is not a member of $V_2$, we introduce
the Cartesian product set
\begin{equation}\label{02eqn:Cartesian product vertex sets}
V=V_1\times (\{o\}\cup V_2)
=\{(i,x)\,;\, i\in V_1,\,\, x\in \{o\}\cup V_2\}.
\end{equation}
Then through the maps $x\mapsto (x,o)$ for $x\in V_1$, and
$y\mapsto (y,i)$ for $y$ being a vertex of a copy of $H$
corresponding to a vertex $i\in V_1$,
we may identify the vertex set of $G\odot H$ with $V$
and, introduce the adjacency relation on $V$ from $G\odot H$.
Thus, the corona $G\odot H$ becomes a graph
on $V$ defined by \eqref{02eqn:Cartesian product vertex sets},
and, from now on, we always take this realization. 

To describe the adjacency relation of the corona graph $G\odot H$
it is convenient to consider the graph join $\Tilde{H}=K_1+H$,
where $K_1$ is the singleton graph on $\{o\}$.
By definition, $\Tilde{H}$ is a graph on $\{o\}\cup V_2$ with
the adjacency relation:
\[
x\sim y \,\,\,\text{in $\Tilde{H}$}
\quad\Longleftrightarrow\quad
\begin{array}{l}
\text{(i) $x\neq o$, $y\neq o$, $x\sim y$ in $H$; or} \\
\text{(ii) $x=o$, $y\neq o$; or} \\
\text{(iii) $x\neq o$, $y=o$}.
\end{array}
\]
Then the corona $G\odot H$ is a graph on
$V=V_1\times (\{o\}\cup V_2)$ 
with the adjacency relation:
\[
(i,x)\sim (j,y)  \,\,\,\text{in $G\odot H$}
\quad\Longleftrightarrow\quad
\begin{array}{l}
\text{(i) $i=j$ and $x\sim y$ in $\Tilde{H}$; or} \\
\text{(ii) $i\sim j$ in $G$ and $x=y=o$}.
\end{array}
\]

Next we focus on the distance matrix of $G\odot H$,
denoted by $D=D_{G\odot H}$.
We note the canonical isomorphism
\begin{align}
C(V)
&=C(V_1\times (\{o\}\cup V_2))
\nonumber \\
&\cong C(V_1)\otimes C(\{o\}\cup V_2) 
\cong C(V_1)\otimes (\mathbb{R}\oplus C(V_2)),
\label{02eqn:C(V) in tensor form}
\end{align}
which follows from \eqref{02eqn:Cartesian product vertex sets}.

\begin{lemma}\label{01lem:distance matrix in tensor form}
Notations and assumptions being as above,
let $D_G$ be the distance matrix of $G$ and 
$A_H$ the adjacency matrix of $H$.
Then, the distance matrix of the corona graph $G\odot H$ is
expressed as 
\begin{equation}\label{02eqn:distance matrix in tensor form}
D_{G\odot H}
=D_G\otimes \begin{bmatrix} J & J \\ J & J \end{bmatrix}
 +I \otimes \begin{bmatrix} 0 & 0 \\ 0 & -2I-A_H \end{bmatrix}
 +J \otimes \begin{bmatrix} 0 & J \\ J & 2J \end{bmatrix},
\end{equation}
in accordance with the canonical isomorphism 
\eqref{02eqn:C(V) in tensor form}.
\end{lemma}

\begin{proof}
Let $d_G$ and $d_{\Tilde{H}}$ be the graph distances of
$G$ and $\Tilde{H}=K_1+H$.
We see easily from definition that
the graph distance of $G\odot H$ is given by
\[
d_{G\odot H}((i,x),(j,y))=
\begin{cases}
d_{\Tilde{H}}(x,y), & \text{if $i=j$}, \\
d_{\Tilde{H}}(x,o)+d_G(i,j)+d_{\Tilde{H}}(o,y), & \text{if $i\neq j$},
\end{cases}
\]
or equivalently,
\begin{align}
d_{G\odot H}((i,x),(j,y))
&=d_G(i,j)J(x,y)+I(i,j)d_{\Tilde{H}}(x,y) 
\nonumber \\
&\qquad  +(J(i,j)-I(i,j))(d_{\Tilde{H}}(x,o)+d_{\Tilde{H}}(o,y)), 
\label{02eqn:in proof Lemma 2.1 (1)}
\end{align}
for any $i,j \in V_1$ and $x,y\in \{o\}\cup V_2$.
Then, letting
\begin{equation}\label{02eqn:in proof Lemma 2.1 (2)}
B_{\Tilde{H}}(x,y)=d_{\Tilde{H}}(x,o)+d_{\Tilde{H}}(o,y),
\qquad x,y\in \{o\}\cup V_2,
\end{equation}
\eqref{02eqn:in proof Lemma 2.1 (1)} becomes
\begin{equation}\label{02eqn:in proof Lemma 2.1 (3)}
D_{G\odot H}
=D_G\otimes J +I \otimes D_{\Tilde{H}}
+(J-I)\otimes B_{\Tilde{H}}\,.
\end{equation}
On the other hand, the graph distance of $\Tilde{H}=K_1+H$ is given by
\begin{equation}\label{02eqn:in proof Lemma 2.1 (4)}
d_{\Tilde{H}}(x,y)=
\begin{cases}
2, & \text{if $x\neq o$, $y\neq o$, $x\neq y$, $x\not\sim y$ in $H$}, \\
1, & \text{if $x\neq o$, $y\neq o$, $x\sim y$ in $H$}, \\
1, & \text{if $x=o$, $y\neq o$ or $x\neq o$, $y=o$}, \\
0, &\text {if $x=y$}.
\end{cases}
\end{equation}
Then the distance matrix $D_{\Tilde{H}}$
and the matrix $B_{\Tilde{H}}$ in \eqref{02eqn:in proof Lemma 2.1 (2)}
are written down in block matrix forms according to
the partition $\{o\}\cup V_2$ as follows:
\begin{equation}\label{02eqn:in proof Lemma 2.1 (5)}
D_{\Tilde{H}}
=\begin{bmatrix}
0 & J \\
J & 2J-2I-A_H
\end{bmatrix},
\qquad
B_{\Tilde{H}}=
\begin{bmatrix}
0 & J \\
J & 2J
\end{bmatrix}.
\end{equation}
Then \eqref{02eqn:distance matrix in tensor form} is
immediate from \eqref{02eqn:in proof Lemma 2.1 (3)}
and \eqref{02eqn:in proof Lemma 2.1 (5)}.
\end{proof}

\begin{remark}\normalfont
An alternative block matrix expression of
the distance matrix of $G\odot H$,
of course, equivalent to \eqref{02eqn:distance matrix in tensor form},
is adopted in \cite{Indulal-Stevanovic2015},
where the purpose is to 
calculate its distance spectrum.
\end{remark}

\subsection{A preliminary formula for $\mathrm{QEC}(G\odot H)$}

We will be concerned with a corona graph $G\odot H$,
where $G$ is a connected graph on two or more vertices.
Namely, we avoid a corona graph $K_1\odot H\cong K_1+H$
from our discussion because $\mathrm{QEC}(K_1)$ is not defined.
In fact, contrary to its simple appearance, 
a general formula for $\mathrm{QEC}(K_1+H)$ is not
known yet, see \cite{Lou-Obata-Huang2022}
for $H$ being a regular graph
and \cite{Mlotkowski-Obata2025a,Mlotkowski-Obata2025b}
for $H=P_n$ being a path.

Thus, from now on we always consider
a corona graph $G\odot H$, where
$G=(V_1,E_1)$ is a connected graph with $|V_1|\ge2$
and $H=(V_2,E_2)$ an arbitrary graph with $|V_2|\ge1$.
We first note a lower bound of $\mathrm{QEC}(G\odot H)$.

\begin{lemma}\label{01lem:lower bound}
Let $G=(V_1,E_1)$ be a connected graph with $|V_1|\ge2$
and $H=(V_2,E_2)$ an arbitrary graph with $|V_2|\ge1$.
Then 
\begin{equation}\label{02eqn:lower bound}
\mathrm{QEC}(G\odot H)
\ge \max\{-2+\sqrt2,\, \mathrm{QEC}(G),\, \mathrm{QEC}(K_1+H)\}.
\end{equation}
\end{lemma}

\begin{proof}
By assumption $G\odot H$ contains a path $P_4\cong K_2\odot K_1$
as an isometric subgraph.
Then, by general theory (see e.g., \cite{Obata-Zakiyyah2018}) we have
\[
\mathrm{QEC}(G\odot H)\ge \mathrm{QEC}(P_4)=-2+\sqrt2.
\]
Similarly, since $G\odot H$ contains $G$ and $K_1+H$ 
as isometric subgraphs, we have
$\mathrm{QEC}(G\odot H)\ge \mathrm{QEC}(G)$ and
$\mathrm{QEC}(G\odot H)\ge \mathrm{QEC}(K_1+H)$.
\end{proof}

\begin{proposition}\label{01prop:QEC=0}
Let $G=(V_1,E_1)$ be a connected graph with $|V_1|\ge2$
and $H=(V_2,E_2)$ an arbitrary graph with $|V_2|\ge1$.
Then $\mathrm{QEC}(G\odot H)=0$ holds if
{\upshape (i)} $\mathrm{QEC}(G)=0$ and $\mathrm{QEC}(K_1+H)\le0$;
or {\upshape (ii)} $\mathrm{QEC}(G)\le0$ and $\mathrm{QEC}(K_1+H)=0$.
\end{proposition}

\begin{proof}
It is known \cite{Mlotkowski-Obata2020,Obata-Zakiyyah2018} that
the star product (vertex amalgamation) of two graphs of QE class
is of QE class.
Since $G\odot H$ is obtained by repeated application of
star product of $G$ and $K_1+H$,
we have $\mathrm{QEC}(G\odot H)\le 0$
if both $G$ and $K_1+H$ are of QE class.
Therefore, $\mathrm{QEC}(G\odot H)\le 0$ holds
if (i) or (ii) is satisfied.
On the other hand, in that case, applying
Lemma \ref{01lem:lower bound}, we obtain
$\mathrm{QEC}(G\odot H)\ge 0$
and hence $\mathrm{QEC}(G\odot H)=0$.
\end{proof}

For simplicity the vertex set of $G\odot H$ is denoted by $V$
and the distance matrix by $D$.
Let $\mathcal{S}=\mathcal{S}_{G\odot H}$ be the set of solutions 
$(f,\lambda,\mu)\in C(V)\times \mathbb{R}\times \mathbb{R}$
to equations:
\begin{align}
(D-\lambda I)f &=\frac{\mu}{2}\bm{1}, 
\label{02eqn:basic equation (1)} \\
\langle \bm{1},f\rangle &=0, 
\label{02eqn:basic equation (2)} \\
\langle f,f\rangle &=1.
\label{02eqn:basic equation (3)}
\end{align}
Then the following assertion is fundamental, see
e.g., \cite{Obata-Zakiyyah2018}.

\begin{proposition}\label{02prop:starting expression for QEC}
Let $G=(V_1,E_1)$ be a connected graph with $|V_1|\ge2$
and $H=(V_2,E_2)$ an arbitrary graph $|V_2|\ge1$.
Then we have
\begin{equation}\label{03eqn:starting formula for QEC(1)}
\mathrm{QEC}(G\odot H)
=\max \lambda(\mathcal{S}),
\end{equation}
where $\lambda(\mathcal{S})$ is the set of $\lambda\in\mathbb{R}$
appearing in a solution $(f,\lambda,\mu)\in\mathcal{S}$.
\end{proposition}

Thus our task is to find the largest $\lambda\in\mathbb{R}$
appearing as a solution $(f,\lambda,\mu)\in\mathcal{S}$ to equations 
\eqref{02eqn:basic equation (1)}--\eqref{02eqn:basic equation (3)}.
If we know in advance that $\mathrm{QEC}(G\odot H)\ge q_0$
for some constant $q_0\in\mathbb{R}$, we only need to look for
the solutions $(f,\lambda,\mu)\in\mathcal{S}$ with $\lambda\ge q_0$
and we have
\begin{equation}\label{02eqn:QEC by lower bound q_0}
\mathrm{QEC}(G\odot H)
=\max \lambda(\mathcal{S})\cap[q_0,+\infty).
\end{equation}
For example, we may take $q_0=-2+\sqrt2$ by 
Lemma \ref{01lem:lower bound}.

To solve equations
\eqref{02eqn:basic equation (1)}--\eqref{02eqn:basic equation (3)}
we adopt the canonical isomorphism
$C(V)\cong C(V_1)\otimes(\mathbb{R}\oplus C(V_2))$ in
\eqref{02eqn:C(V) in tensor form}.
Let $\{e_x\,;\, x\in V_2\}$ be the canonical basis of $C(V_2)$.
Since the vectors 
\begin{equation}\label{02eqn:canonical base}
\begin{bmatrix} 1 \\ 0 \end{bmatrix},
\qquad
\begin{bmatrix} 0 \\ e_x \end{bmatrix},
\quad x\in V_2,
\end{equation}
from the canonical basis of $\mathbb{R}\oplus C(V_2)$,
any $f\in C(V) \cong C(V_1)\otimes (\mathbb{R}\oplus C(V_2))$
admits a unique expression of the form:
\begin{equation}\label{02eqn:expansion of f in C(V)}
f=\xi_o\otimes \begin{bmatrix} 1 \\ 0 \end{bmatrix}
+ \sum_{x\in V_2} \xi_x\otimes \begin{bmatrix} 0 \\ e_x \end{bmatrix},
\qquad
\xi_x \in C(V_1),
\quad x\in \{o\}\cup V_2\,.
\end{equation}
Thus, equations 
\eqref{02eqn:basic equation (1)}--\eqref{02eqn:basic equation (3)}
for $(f,\lambda,\mu)$ 
are naturally transformed into equivalent equations
for $(\xi_o,\{\xi_x\,;\,x\in V_2\},\lambda,\mu)$.
 
\begin{lemma}\label{02lem:equivalent form of 2.10}
Equation \eqref{02eqn:basic equation (1)} is equivalent to
the following set of equations:
\begin{align}
&(D_G-\lambda I)\xi_o 
+\sum_{x\in V_2}(D_G\xi_x+J\xi_x)=\frac{\mu}{2}\,\bm{1},
\label{02eqn:basic equation (011)}
\\
&(J+\lambda I)\xi_o +\sum_{y\in V_2} J\xi_y
-\sum_{y\in V_2} \langle e_x, (A_H+2+\lambda)e_y\rangle \xi_y=0,
\qquad x\in V_2,
\label{02eqn:basic equation (012)}
\end{align}
\end{lemma}

\begin{proof}
Using
\begin{align*}
D\bigg(\xi_o\otimes \begin{bmatrix} 1 \\ 0 \end{bmatrix}\bigg)
&=D_G\xi_o\otimes \begin{bmatrix} 1 \\ \mathbf{1} \end{bmatrix}
 +J\xi_o\otimes \begin{bmatrix} 0 \\ \mathbf{1} \end{bmatrix},
\\
D\bigg(\xi_x\otimes \begin{bmatrix} 0 \\ e_x \end{bmatrix}\bigg)
&=D_G\xi_x\otimes \begin{bmatrix} 1 \\ \mathbf{1} \end{bmatrix}
 +\xi_x\otimes \begin{bmatrix} 0 \\ -(A_H+2)e_x \end{bmatrix}
 +J\xi_x\otimes \begin{bmatrix} 1 \\ 2\mathbf{1} \end{bmatrix},
\quad x\in V_2,
\end{align*}
which follow easily by 
Lemma \ref{01lem:distance matrix in tensor form},
equation \eqref{02eqn:basic equation (1)} is brought into
\begin{align*}
&D_G\xi_o\otimes \begin{bmatrix} 1 \\ \bm{1} \end{bmatrix}
 +J\xi_o\otimes \begin{bmatrix} 0 \\ \bm{1} \end{bmatrix}
 +\sum_{x\in V_2} 
  D_G\xi_x\otimes \begin{bmatrix} 1 \\ \mathbf{1} \end{bmatrix}
 +\sum_{x\in V_2} 
  \xi_x\otimes \begin{bmatrix} 0 \\ -(A_H+2)e_x \end{bmatrix} 
\\
&\qquad \qquad
 +\sum_{x\in V_2} 
  J\xi_x\otimes \begin{bmatrix} 1 \\ 2\mathbf{1} \end{bmatrix}
-\lambda \xi_o\otimes \begin{bmatrix} 1 \\ 0 \end{bmatrix}
-\sum_{x\in V_2} 
  \lambda \xi_x\otimes \begin{bmatrix} 0 \\ e_x \end{bmatrix}
=\frac{\mu}{2}\,
 \mathbf{1}\otimes \begin{bmatrix} 1 \\ \mathbf{1} \end{bmatrix}.
\end{align*}
Comparing the coefficients with respect to the canonical basis
\eqref{02eqn:canonical base}, 
equation \eqref{02eqn:basic equation (1)} is equivalently
brought into the set of equations:
\begin{align}
&(D_G-\lambda I)\xi_o 
+\sum_{x\in V_2}(D_G\xi_x+J\xi_x)=\frac{\mu}{2}\,\bm{1},
\label{03eqn:main equation (01)}\\
&(D_G+J)\xi_o 
+\sum_{y\in V_2}D_G\xi_y
+\sum_{y\in V_2}2J\xi_y
\nonumber \\
&\qquad\qquad +\sum_{y\in V_2} \langle e_x, -(A_H+2)e_y\rangle \xi_y
-\lambda\xi_x
=\frac{\mu}{2}\,\mathbf{1},
\qquad x\in V_2.
\label{03eqn:main equation (02)} 
\end{align}
Since \eqref{02eqn:basic equation (012)} is obtained
by subtracting \eqref{03eqn:main equation (01)} from
\eqref{03eqn:main equation (02)},
the set of equations \eqref{03eqn:main equation (01)} and
\eqref{03eqn:main equation (02)} are equivalent to
\eqref{02eqn:basic equation (1)}.
\end{proof}

\begin{proposition}\label{03prop:reformulated basic equations}
The set $\lambda(\mathcal{S})$ coincides with
the set of $\lambda\in\mathbb{R}$ appearing
as a solution $(\xi_o,\{\xi_x\,;\, x \in V_2\}, \lambda, \mu)$
to the following equations:
\begin{align}
&(D_G-J-\lambda)\xi_o 
+\sum_{x\in V_2}D_G\xi_x=\frac{\mu}{2}\,\mathbf{1},
\label{02eqn:basic equation (11)}
\\
&\lambda\xi_o
=\sum_{y\in V_2} \langle e_x, (A_H+2+\lambda)e_y\rangle \xi_y\,,
\qquad x\in V_2,
\label{02eqn:basic equation (12)}
\\
&\langle \bm{1}, \xi_o\rangle
+\sum_{x\in V_2} \langle \bm{1}, \xi_x\rangle=0.
\label{02eqn:basic equation (13)}
\\
&\langle \xi_o, \xi_o\rangle
+\sum_{x\in V_2} \langle \xi_x, \xi_x\rangle=1.
\label{02eqn:basic equation (14)}
\end{align}
\end{proposition}

\begin{proof}
Using \eqref{02eqn:expansion of f in C(V)}, equations
\eqref{02eqn:basic equation (2)} and \eqref{02eqn:basic equation (3)}
are equivalently brought into
\eqref{02eqn:basic equation (13)}
and \eqref{02eqn:basic equation (14)}, respectively.
Then, we see from  Lemma \ref{02lem:equivalent form of 2.10}
that equations 
\eqref{02eqn:basic equation (1)}--\eqref{02eqn:basic equation (3)}
are equivalent to
\eqref{02eqn:basic equation (011)},
\eqref{02eqn:basic equation (012)},
\eqref{02eqn:basic equation (13)} and 
\eqref{02eqn:basic equation (14)}.
Here we note that \eqref{02eqn:basic equation (13)} is equivalent to
\begin{equation}\label{03eqn:main equation (3)}
\sum_{x\in V_2}J\xi_x
=\sum_{x\in V_2}\langle \mathbf{1},\xi_x \rangle\mathbf{1}
=-\langle \mathbf{1},\xi_o \rangle\mathbf{1}
=-J\xi_o\,.
\end{equation}
Then, \eqref{02eqn:basic equation (11)}
and \eqref{02eqn:basic equation (12)} follow
by applying \eqref{03eqn:main equation (3)}
to \eqref{02eqn:basic equation (011)} and
\eqref{02eqn:basic equation (012)}, respectively.
Equations
\eqref{02eqn:basic equation (1)}--\eqref{02eqn:basic equation (3)}
being transformed into 
\eqref{02eqn:basic equation (11)}--\eqref{02eqn:basic equation (14)},
we see that $\lambda(\mathcal{S})$ coincides with
the set of $\lambda\in\mathbb{R}$ appearing
as a solution $(\xi_o,\{\xi_x\,;\, x\in V_2\}, \lambda, \mu)$
to equations
\eqref{02eqn:basic equation (11)}--\eqref{02eqn:basic equation (14)}.
\end{proof}

\begin{remark}\normalfont
Equation \eqref{02eqn:basic equation (1)} with $\mu=0$ is
nothing else but the eigen-equation associated with $D$.
Then, we see from Lemma \ref{02lem:equivalent form of 2.10}
that the eigenvalues $\lambda$ of $D$ are obtained
from the solutions $(\lambda,\xi_o,\{\xi_x\,;\,x\in V_2\})$ to
the equations \eqref{02eqn:basic equation (011)}
and \eqref{02eqn:basic equation (012)} with $\mu=0$.
The distance eigenvalues of a corona graph $G\odot H$ of
special type are determined in \cite{Indulal-Stevanovic2015}
by solving the eigen-equation in a form essentially
the same as \eqref{02eqn:basic equation (011)}
and \eqref{02eqn:basic equation (012)} with $\mu=0$.
\end{remark}

\section{Two Key Functions Associated with a Symmetric Matrix}

This section is devoted to general consideration.
With a symmetric matrix $A\in M_n(\mathbb{R})$, $n\ge1$, we
associate two real functions defined by
\begin{align}
\omega_A(\lambda)
&= 1+\lambda \langle\mathbf{1},(A+2+\lambda)^{-1}\mathbf{1}\rangle,
\label{03eqn:def of omega(lambda)}
\\[6pt]
\psi_A(\lambda)
&=\frac{\lambda}{\omega_A(\lambda)},
\label{03eqn:def of psi(lambda) original}
\end{align}
of which the precise domains and properties 
will be clarified in the discussion below.

\subsection{Main eigenvalues}
\label{03sec:Main eigenvalues}

Let $A\in M_n(\mathbb{R})$, $n\ge1$, be a symmetric matrix.
The set of distinct eigenvalues of $A$ is denoted by $\mathrm{ev}(A)$.
For an eigenvalue $\alpha\in\mathrm{ev}(A)$ we denote by
$E_\alpha$ the orthogonal projection onto the 
associated eigenspace.
By the spectral decomposition we come to
\begin{equation}\label{03eqn:spectral decomposition of A}
A=\sum_{\alpha\in\mathrm{ev}(A)} \alpha E_\alpha\,,
\qquad
\sum_{\alpha\in\mathrm{ev}(A)}E_\alpha =I,
\qquad
E_\alpha E_{\alpha^\prime}=\delta_{\alpha\alpha^\prime} E_\alpha\,.
\end{equation}
By convention we set $E_\alpha=O$ for $\alpha\not\in\mathrm{ev}(A)$.

Following Cvetkovi\'c \cite{Cvetkovic1978,Cvetkovic-Rowlinson-Simic2010}
an eigenvalue $\alpha\in\mathrm{ev}(A)$ is called
a \textit{main eigenvalue} if
there exists a non-zero vector $v\in \mathbb{R}^n$ such that
$Av=\alpha v$ and $\langle \mathbf{1}, v\rangle\neq0$,
or equivalently if $E_\alpha\mathbf{1}\neq0$.
Let $\mathrm{ev}_{\mathrm{M}}(A)$ denote the set of main eigenvalues of $A$.
It follows from the second relation in
\eqref{03eqn:spectral decomposition of A} that
$\mathrm{ev}_{\mathrm{M}}(A)\neq\emptyset$.

\subsection{The $\omega$-function associated with a symmetric matrix}

Applying the spectral decomposition 
\eqref{03eqn:spectral decomposition of A} to
\eqref{03eqn:def of omega(lambda)}, we obtain
\begin{align}
\omega_A(\lambda)
&=1+\sum_{\alpha\in\mathrm{ev}_{\mathrm{M}}(A)} 
 \frac{\lambda}{\alpha+2+\lambda} \,\|E_\alpha\mathbf{1}\|^2
\nonumber \\
&=1+\|E_{-2}\mathbf{1}\|^2
 +\sum_{\alpha\in\mathrm{ev}_{\mathrm{M}}(A)\backslash\{-2\}} 
 \frac{\lambda}{\alpha+2+\lambda} \,\|E_\alpha\mathbf{1}\|^2,
\label{03eqn:omega(lambda) (21)}
\end{align}
where $\|E_{-2}\mathbf{1}\|>0$ if $-2\in\mathrm{ev}_M(A)$
and $\|E_{-2}\mathbf{1}\|=0$ otherwise.
As usual, the last sum is understood to be zero when
$\mathrm{ev}_{\mathrm{M}}(A)\backslash\{-2\}=\emptyset$.
We observe that the last expression in \eqref{03eqn:omega(lambda) (21)}
is a real analytic function on the domain
\begin{equation}\label{03eqn:domain of omega}
\mathbb{R}\backslash
\{-\alpha-2\,;\,\alpha\in\mathrm{ev}_{\mathrm{M}}(A)\backslash\{-2\}\},
\end{equation}
which has a simple pole at $\lambda=-\alpha-2$ for each
$\alpha\in\mathrm{ev}_{\mathrm{M}}(A)\backslash\{-2\}$.
From now on, we understand $\omega_A(\lambda)$ as defined in this way.
In other words, $\omega_A(\lambda)$ is the unique maximal
analytic extension satisfying \eqref{03eqn:def of omega(lambda)}
for $\lambda\in\mathbb{R}$ with $\det(A+2+\lambda)\neq0$.
For convenience we call $\omega_A(\lambda)$ the 
\textit{$\omega$-function} associated with a symmetric matrix $A$.
The number of poles of $\omega_A(\lambda)$ is denoted by
\begin{equation}\label{03eqn:def of k_A}
k_A=|\mathrm{ev}_{\mathrm{M}}(A)\backslash\{-2\}|.
\end{equation}

Since $\mathrm{ev}_{\mathrm{M}}(A)\neq\emptyset$ in general,
we see that $k_A=0$ if and only if $\mathrm{ev}_{\mathrm{M}}(A)=\{-2\}$.
In that case, from \eqref{03eqn:omega(lambda) (21)} we have
\begin{equation}\label{03eqn:omega(lambda)=const}
\omega_A(\lambda)
=1+\|E_{-2}\mathbf{1}\|^2
=1+\|\mathbf{1}\|^2
=1+n\quad\text{(constant)}.
\end{equation}
For example, $A=-(2/n)J_n\in M_n(\mathbb{R})$ falls
into this situation.
We next note that $k_A\ge1$ is equivalent to
$\mathrm{ev}_{\mathrm{M}}(A)\backslash\{-2\}\neq\emptyset$.
Since $\mathrm{ev}_{\mathrm{M}}(A)\neq\emptyset$ in general,
$k_A\ge1$ is also equivalent to $\mathrm{ev}_{\mathrm{M}}(A)\neq\{-2\}$.
In that case, we see directly from \eqref{03eqn:omega(lambda) (21)}
that $\omega_A(\lambda)$ is a rational function of the form:
\begin{equation}\label{03eqn:omega(lambda) as rational function}
\omega_A(\lambda)
=p(\lambda)\prod_{\alpha\in\mathrm{ev}_{\mathrm{M}}(A)\backslash\{-2\}}
 (\alpha+2+\lambda)^{-1},
\end{equation}
where $p(\lambda)$ is a polynomial of order $k_A$ with
the leading coefficient $n+1$.

\begin{lemma}\label{03lem:behaviour of omega(lambda)}
Let $A\in M_n(\mathbb{R})$ be a symmetric matrix
and $\omega_A(\lambda)$ the associated $\omega$-function
defined as above. 
\begin{enumerate}
\item[\upshape (1)] 
$\omega_A(0)=1+\|E_{-2}\mathbf{1}\|^2$.
\item[\upshape (2)] 
$\displaystyle\lim_{\lambda\rightarrow\pm\infty}\omega_A(\lambda)=n+1$.
\item[\upshape (3)]
For $\alpha\in\mathrm{ev}_{\mathrm{M}}(A)\backslash\{-2\}$ with $\alpha+2>0$
we have
\[
\lim_{\lambda\downarrow -\alpha-2}\omega_A(\lambda)=-\infty,
\qquad
\lim_{\lambda\uparrow -\alpha-2}\omega_A(\lambda)=+\infty.
\]
\item[\upshape (4)]
For $\alpha\in\mathrm{ev}_{\mathrm{M}}(A)\backslash\{-2\}$ with $\alpha+2<0$
we have
\[
\lim_{\lambda\downarrow -\alpha-2}\omega_A(\lambda)=+\infty,
\qquad
\lim_{\lambda\uparrow -\alpha-2}\omega_A(\lambda)=-\infty.
\]
\end{enumerate}
\end{lemma}

\begin{proof}
Straightforward from the expression in
\eqref{03eqn:omega(lambda) (21)}.
\end{proof}

We will examine the zeroes of $\omega_A(\lambda)$.
From \eqref{03eqn:def of k_A} we see that
\[
|\mathrm{ev}_{\mathrm{M}}(A)|=
\begin{cases}
k_A, & \text{if $-2\not\in \mathrm{ev}_{\mathrm{M}}(A)$}, \\
k_A+1, & \text{if $-2\in \mathrm{ev}_{\mathrm{M}}(A)$}.
\end{cases}
\]
If $k_A=0$, or equivalently if
$\mathrm{ev}_{\mathrm{M}}(A)=\{-2\}$,
then $\omega_A(\lambda)$ has no zero 
by \eqref{03eqn:omega(lambda)=const}.

\begin{lemma}\label{03lem:zeroes of omega(lambda) (1)}
Assume that $k_A\ge1$ and $-2\not\in\mathrm{ev}_{\mathrm{M}}(A)$.
Arrange the members in $\mathrm{ev}_{\mathrm{M}}(A)\cup \{-2\}$
in ascending order as
\[
\mathrm{ev}_{\mathrm{M}}(A)\cup \{-2\}
=\{\alpha_1<\alpha_2<\dots<\alpha_{k_A+1}\}.
\]
Then, for each $1\le i\le k_A$ the interval
$(-\alpha_{i+1}-2,-\alpha_i-2)$
contains a unique zero of $\omega_A(\lambda)$.
Moreover, those $k_A$ distinct real zeroes
exhaust the zeroes of $\omega_A(\lambda)$.
\end{lemma}

\begin{proof}
For simplicity we write $k=k_A$.
Since $\omega_A(\lambda)$ is a rational function
as in \eqref{03eqn:omega(lambda) as rational function}
with $p(\lambda)$ being a polynomial of order $k$,
$\omega_A(\lambda)$ has at most $k$ distinct zeroes
in the complex plane.
Therefore, for the assertion it is sufficient to show that
each interval $(-\alpha_{i+1}-2,-\alpha_i-2)$
for $1\le i\le k$
contains at least one zero of $\omega_A(\lambda)$.

(Case 1) $\alpha_1=-2$. 
In this case, we have $-\alpha_k-2<\dotsb<-\alpha_2-2<0$.
By definition,
$\omega_A(\lambda)$ is continuous on the interval
$(-\alpha_{i+1}-2,-\alpha_i-2)$ for each $2\le i\le k$.
Moreover, 
we see from Lemma \ref{03lem:behaviour of omega(lambda)} (3) that
\[
\lim_{\lambda\downarrow -\alpha_{i+1}-2}\omega_A(\lambda)=-\infty,
\qquad
\lim_{\lambda\uparrow -\alpha_i-2}\omega_A(\lambda)=+\infty.
\]
Therefore, $\omega_A(\lambda)$ has at least one zero
in the interval $(-\alpha_{i+1}-2,-\alpha_i-2)$ for each $2\le i\le k$.
Moreover, since 
$\omega_A(\lambda)$ is a continuous function on
$(-\alpha_2-2,+\infty)$ and
\[
\lim_{\lambda\downarrow -\alpha_2-2}\omega(\lambda)=-\infty,
\qquad
\omega_A(0)=1,
\]
it has at least one zero
in the interval $(-\alpha_2-2,0)=(-\alpha_2-2,-\alpha_1-2)$.

(Case 2) $\alpha_k=-2$. In a similar manner as in (Case 1),
we may show that
for each $1\le i\le k$ the interval $(-\alpha_{i+1}-2,-\alpha_i-2)$
contains at least one zero of $\omega_A(\lambda)$.

(Case 3) $\alpha_1<-2<\alpha_k$. In that case we have
$\alpha_l=-2$ for some $2\le l \le k-1$.
For $1\le i\le l-2$ the function $\omega_A(\lambda)$ is continuous
on $(-\alpha_{i+1}-2,-\alpha_i-2)$ and
\[
\lim_{\lambda\downarrow -\alpha_{i+1}-2}\omega_A(\lambda)=+\infty,
\qquad
\lim_{\lambda\uparrow -\alpha_i-2}\omega_A(\lambda)=-\infty.
\]
by Lemma \ref{03lem:behaviour of omega(lambda)} (4).
Therefore, $\omega_A(\lambda)$ has at least one zero
in the $(-\alpha_{i+1}-2,-\alpha_i-2)$ for $1\le i\le l-2$.
Similarly, $\omega_A(\lambda)$ has at least one zero
in the $(-\alpha_{i+1}-2,-\alpha_i-2)$ for
$l+1\le i\le k$.
Finally, since $\omega_A(\lambda)$ is continuous
on $(-\alpha_{l+1}-2,-\alpha_{l-1}-2)$ and
\[
\lim_{\lambda\downarrow -\alpha_{l+1}-2}\omega_A(\lambda)=-\infty,
\qquad
\omega_A(0)=1,
\qquad
\lim_{\lambda\uparrow -\alpha_{l-1}-2}\omega_A(\lambda)=-\infty
\]
by Lemma \ref{03lem:behaviour of omega(lambda)},
it has at least one zero
in each of the intervals
$(-\alpha_{l+1}-2,0)=(-\alpha_{l+1}-2,-\alpha_l-2)$
and $(0,-\alpha_{l-1}-2)=(-\alpha_l-2,-\alpha_{l-1}-2)$.
Consequently, $\omega_A(\lambda)$ 
has at least one zero in each interval $(-\alpha_{i+1}-2,-\alpha_i-2)$
for $1\le i\le k$.
\end{proof}

\begin{lemma}\label{03lem:zeroes of omega(lambda) (2)}
Assume that $k_A\ge1$ and $-2\in\mathrm{ev}_{\mathrm{M}}(A)$.
Arrange the members in $\mathrm{ev}_{\mathrm{M}}(A)$
in ascending order as
\[
\mathrm{ev}_{\mathrm{M}}(A)
=\{\alpha_1<\alpha_2<\dots<\alpha_{k_A+1}\}.
\]
Then, for each $1\le i\le k_A$ the interval
$(-\alpha_{i+1}-2,-\alpha_i-2)$
contains a unique zero of $\omega_A(\lambda)$.
Moreover, those $k_A$ distinct real zeroes
exhaust the zeroes of $\omega_A(\lambda)$.
\end{lemma}

\begin{proof}
We first note that $\omega_A(\lambda)$ has at most $k_A$
distinct zeroes in the complex plane by
\eqref{03eqn:omega(lambda) as rational function}.
Therefore, for the assertion it is sufficient to show that
for each $1\le i\le k_A$ the interval $(-\alpha_{i+1}-2,-\alpha_i-2)$
contains at least one zero of $\omega_A(\lambda)$.
In fact, this claim is verified by using similar argument as
in the proof of Lemma
\ref{03lem:zeroes of omega(lambda) (1)}.
\end{proof}

Combining Lemmas \ref{03lem:zeroes of omega(lambda) (1)} 
and \ref{03lem:zeroes of omega(lambda) (2)},
we come to the following 

\begin{proposition}\label{03prop:zeroes of omega(lambda) (3)}
Let $A\in M_n(\mathbb{R})$, $n\ge1$, be a symmetric matrix
and $\omega_A(\lambda)$ the associated $\omega$-function.
Assume that $k_A=|\mathrm{ev}_{\mathrm{M}}(A)\backslash\{-2\}|\ge1$
and arrange the members of $\mathrm{ev}_{\mathrm{M}}(A)\cup \{-2\}$
in ascending order as
\begin{equation}\label{03eqn:ascending order alpha}
\mathrm{ev}_{\mathrm{M}}(A)\cup \{-2\}
=\{\alpha_1<\alpha_2<\dots<\alpha_{k_A+1}\}.
\end{equation}
Then, for each $1\le i\le k_A$ the interval
$(-\alpha_{i+1}-2,-\alpha_i-2)$
contains a unique zero of $\omega_A(\lambda)$.
Moreover, those $k_A$ distinct zeroes exhaust the zeroes of
$\omega_A(\lambda)$.
\end{proposition}

We are also interested in the largest zero of $\omega_A(\lambda)$.

\begin{proposition}\label{03prop:max zero of omega}
We keep the notations and assumptions in Proposition
\ref{03prop:zeroes of omega(lambda) (3)}.
Let $\lambda_*$ be the largest zero of $\omega_A(\lambda)$.
Then, we have $\lambda_*<0$ if $\min\mathrm{ev}_{\mathrm{M}}(A)\ge -2$,
and $\lambda_*>0$ otherwise.
\end{proposition}

\begin{proof}
It follows from Proposition \ref{03prop:zeroes of omega(lambda) (3)}
that
\begin{equation}\label{03eqn:max zero lambda*}
-\alpha_{2}-2<\lambda_*<-\alpha_1-2,
\end{equation}
where $\mathrm{ev}_{\mathrm{M}}(A)\cup \{-2\}
=\{\alpha_1<\alpha_2<\dots<\alpha_{k_A+1}\}$.
If $\min\mathrm{ev}_{\mathrm{M}}(A)\ge -2$,
then $\alpha_1=-2$ and we come to $\lambda_*<-\alpha_1-2=0$
by \eqref{03eqn:max zero lambda*}.
If $\min\mathrm{ev}_{\mathrm{M}}(A)<-2$,
then we have
$\min\mathrm{ev}_{\mathrm{M}}(A)=\alpha_1<\alpha_2\le -2$
and $\lambda_*>-\alpha_{2}-2\ge0$ by \eqref{03eqn:max zero lambda*}.
\end{proof}

\subsection{The $\psi$-function and its inverse function}

Let $A\in M_n(\mathbb{R})$, $n\ge1$, be a symmetric matrix
and $\omega_A(\lambda)$ the associated $\omega$-function 
introduced in the previous section.
The \textit{$\psi$-function} is defined by
\begin{equation}\label{03eqn:def of psi(lambda)}
\psi_A(\lambda)=\frac{\lambda}{\omega_A(\lambda)}\,.
\end{equation}
Obviously, $\psi_A(\lambda)$ is a real analytic function on the domain
\begin{equation}\label{03eqn:original domain of psi}
\mathbb{R}\backslash 
(\{-\alpha-2\,;\alpha\in \mathrm{ev}_{\mathrm{M}}(A)\backslash\{-2\}\}
\cup \{\omega_A(\lambda)=0\}).
\end{equation}

If $k_A=0$, that is, $\mathrm{ev}_{\mathrm{M}}(A)=\{-2\}$, 
then \eqref{03eqn:original domain of psi} is reduced to
the whole  $\mathbb{R}$ and \eqref{03eqn:def of psi(lambda)} becomes
\begin{equation}\label{03eqn:psi for evm=-2}
\psi_A(\lambda)
=\frac{\lambda}{\omega_A(\lambda)}
=\frac{\lambda}{n+1}\,.
\end{equation}
Assume that $k_A\ge1$, that is, 
$\mathrm{ev}_{\mathrm{M}}(A)\neq\{-2\}$.
For the exclusion sets of \eqref{03eqn:original domain of psi},
by Proposition \ref{03prop:zeroes of omega(lambda) (3)}
we see that
\begin{equation}\label{03eqn:disjoint exclusion sets}
\{-\alpha-2\,;\alpha\in \mathrm{ev}_{\mathrm{M}}(A)\backslash\{-2\}\}
\cap \{\omega_A(\lambda)=0\}=\emptyset,
\end{equation}
Moreover, we see from Lemma \ref{03lem:behaviour of omega(lambda)} that
\[
\lim_{\lambda\rightarrow -\alpha-2}\psi_A(\lambda)=0,
\qquad
\alpha\in\mathrm{ev}_{\mathrm{M}}(A)\backslash\{-2\}.
\]
Therefore, together with \eqref{03eqn:disjoint exclusion sets}
we see that $\psi_A(\lambda)$ is extended uniquely to
a real analytic function on
$\mathbb{R}\backslash \{\omega_A(\lambda)=0\}$.
From now on, we understand that $\psi_A(\lambda)$ is the
unique real analytic function on
$\mathbb{R}\backslash \{\omega_A(\lambda)=0\}$ which
satisfies \eqref{03eqn:def of psi(lambda)}.
In particular, we have
\begin{equation}\label{03eqn:zeroes of psi}
\psi_A(0)=0,
\qquad
\psi_A(-\alpha-2)=0, 
\quad \alpha\in\mathrm{ev}_{\mathrm{M}}(A)\backslash\{-2\},
\end{equation}
and $\psi_A(\lambda)$ has a simple pole at
each zero of $\omega_A(\lambda)$.

\begin{proposition}\label{03prop:behaviour of psi(lambda)}
Let $A\in M_n(\mathbb{R})$, $n\ge1$, be a symmetric matrix
and assume that $k_A=|\mathrm{ev}_{\mathrm{M}}(A)\backslash\{-2\}|\ge1$.
Let $\omega_A(\lambda)$ and $\psi_A(\lambda)$ be the
associated $\omega$- and $\psi$-functions, respectively.
Let $\lambda_1>\lambda_2>\dotsb>\lambda_{k_A}$ be the zeroes of
$\omega_A(\lambda)$.
Then $\psi_A(\lambda)$ is a continuous and 
strictly increasing function on
$\mathbb{R}\backslash\{\lambda_1,\lambda_2,\dots, \lambda_{k_A}\}$ and
\[
\lim_{\lambda\rightarrow+\infty}\psi_A(\lambda)=+\infty,
\quad
\lim_{\lambda\downarrow \lambda_i}\psi_A(\lambda)=-\infty,
\quad
\lim_{\lambda\uparrow \lambda_i}\psi_A(\lambda)=+\infty,
\quad
\lim_{\lambda\rightarrow-\infty}\psi_A(\lambda)=-\infty,
\]
for $1\le i\le k_A$.
In particular, $\psi_A(\lambda)$ splits into $k_A+1$ 
continuous increasing functions on the intervals
$(-\infty,\lambda_{k_A})$, $(\lambda_{k_A},\lambda_{k_A-1}),
\dots (\lambda_2,\lambda_1)$, $(\lambda_1, +\infty)$.
\end{proposition}

\begin{proof}
By differentiating $\psi_A(\lambda)$ directly, we obtain
\[
\frac{d}{d\lambda}\,\psi_A(\lambda)
=\frac{1}{\omega_A(\lambda)^2}
  \Big\{1+\lambda^2\|(A+2+\lambda)^{-1}\mathbf{1}\|^2 \Big\}>0.
\]
Therefore, $\psi_A(\lambda)$ is 
strictly increasing on $\mathbb{R}\backslash\{\omega_A(\lambda)=0\}$.
The limit behavior follows from
Lemma \ref{03lem:behaviour of omega(lambda)} and
Proposition \ref{03prop:zeroes of omega(lambda) (3)}.
\end{proof}

It follows from Proposition \ref{03prop:behaviour of psi(lambda)} that
the inverse function $\psi_A^{-1}$ is determined only as
a multivalued function (in fact, taking exactly $(k_A+1)$ values).
Let $\lambda_*=\lambda_1$ be the largest zero of $\omega_A(\lambda)$
and define $\psi_{A*}=\psi\!\restriction_{(\lambda_*,+\infty)}$.
Since $\psi_{A*}(\lambda)$ becomes a real analytic
and strictly increasing function
on $(\lambda_*,\infty)$ which varies from $-\infty$ to $+\infty$,
its inverse function 
$\psi_{A*}^{-1}(\lambda^\prime)$ is well defined
and becomes a real analytic and strictly increasing function.

\begin{proposition}\label{03prop:behaviour of psi inverse}
We keep the notations and assumptions as in
Proposition \ref{03prop:behaviour of psi(lambda)}.
\begin{enumerate}
\item[\upshape (1)] $\psi_{A*}^{-1}(\lambda^\prime)$ 
is a real analytic and strictly increasing function
on $\mathbb{R}$ such that
\[
\lim_{\lambda^\prime\rightarrow-\infty} 
\psi_{A*}^{-1}(\lambda^\prime)=\lambda_*,
\qquad
\lim_{\lambda^\prime\rightarrow+\infty} 
\psi_{A*}(\lambda^\prime)=+\infty.
\]
\item[\upshape (2)] $\psi_{A*}^{-1}(0)=-\alpha_1-2$,
where $\alpha_1=\min \mathrm{ev}_{\mathrm{M}}(A)\cup\{-2\}$.
\item[\upshape (3)] 
$\psi_{A*}^{-1}(\lambda^\prime)
=\max \{\lambda\in\mathbb{R}\,;\, \psi_A(\lambda)=\lambda^\prime\}$.
\end{enumerate}
\end{proposition}

\begin{proof}
(1) Already clear as mentioned above.

(2) We note from $-\alpha_2-2<\lambda_*<-\alpha_1-2$ that
$-\alpha_1-2$ belongs to the domain of $\psi_{A*}(\lambda)$.
Then $\psi_{A*}(-\alpha_1-2)=0$ by \eqref{03eqn:zeroes of psi}
and we have $\psi_{A*}^{-1}(0)=-\alpha_1-2$.

(3) Given $\lambda^\prime\in\mathbb{R}$,
there exist $k_A+1$ distinct $\lambda\in \mathbb{R}$
satisfying $\psi_A(\lambda)=\lambda^\prime$ by
Proposition \ref{03prop:behaviour of psi(lambda)}.
Since $\psi_{A*}(\lambda)$ is chosen as the most right
branch of $\psi_A(\lambda)$,
we see that $\psi_{A*}^{-1}(\lambda^\prime)$ coincides with
the largest $\lambda$ satisfying $\psi_A(\lambda)=\lambda^\prime$.
\end{proof}

\begin{remark}
\normalfont
The assertion of Proposition \ref{03prop:behaviour of psi inverse}
remains valid also for $k_A=0$ under the convention that
$\lambda_*=-\infty$.
\end{remark}

\section{Main Formula for $\mathrm{QEC}(G\odot H)$}

\subsection{The $\omega$- and $\psi$-functions associated with a graph}

We first note the following

\begin{lemma}\label{04lem:main eigen is only -2}
Let $H$ be an arbitrary graph and $A_H$ the adjacency matrix.
\begin{enumerate}
\item[\upshape (1)] $\mathrm{ev}_{\mathrm{M}}(A_H)\neq\emptyset$
and $\mathrm{ev}_{\mathrm{M}}(A_H)\neq\{-2\}$.
\item[\upshape (2)] 
$\max\mathrm{ev}(A_H)\in \mathrm{ev}_{\mathrm{M}}(A_H)$,
i.e., the largest eigenvalue of $A_H$ is a main one.
\end{enumerate}
\end{lemma}

\begin{proof}
(1) It is pointed out in Section \ref{03sec:Main eigenvalues}
that $\mathrm{ev}_{\mathrm{M}}(A_H)\neq\emptyset$.
If $\mathrm{ev}_{\mathrm{M}}(A)=\{-2\}$, 
we have $A\mathbf{1}=-2\mathbf{1}$.
But this contradicts that any row sum of $A_H$ is non-negative.

(2) By Perron-Frobenius theorem,
associated with the largest eigenvalue there exists
an eigenvector $f$ of which the entries are all non-negative.
Since $f$ is not orthogonal to $\bm{1}$,
the largest eigenvalue is a main eigenvalue,
see also \cite{Cvetkovic1978,Cvetkovic-Rowlinson-Simic2010}.
\end{proof}

Let $\omega_H(\lambda)$ be the $\omega$-function associated with
the adjacency matrix $A_H$ of a given graph $H$.
Following the argument in the previous section,
we know that $\omega_H(\lambda)$ is a real analytic function on
$\mathbb{R}\backslash
\{-\alpha-2\,;\, 
\alpha\in \mathrm{ev}_{\mathrm{M}}(A_H)\backslash\{-2\}\}$
satisfying
\begin{equation}\label{03eqn:def of omega(lambda) for H}
\omega_H(\lambda)
= 1+\lambda \langle\mathbf{1},(A_H+2+\lambda)^{-1}\mathbf{1}\rangle.
\end{equation}
By Lemma \ref{04lem:main eigen is only -2} (1) we have
\[
k_H=|\mathrm{ev}_{\mathrm{M}}(A_H)\backslash\{-2\}|\ge1.
\]
Moreover, we see 
from Proposition \ref{03prop:zeroes of omega(lambda) (3)} that
$\omega_H(\lambda)$ has $k_H$ distinct real zeroes,
say, $\lambda_1>\lambda_2>\dots>\lambda_{k_H}$.
Then we have
\begin{equation}\label{04eqn:interlacing of lambda and -2-alpha}
-\alpha_{i+1}-2<\lambda_i<-\alpha_{i}-2,
\qquad 1\le i\le k_H\,,
\end{equation}
where $\mathrm{ev}_{\mathrm{M}}(A_H)\cup \{-2\}
=\{\alpha_1<\alpha_2<\dots<\alpha_{k_H+1}\}$.

Let $\psi_H(\lambda)$ be the $\psi$-function associated with $A_H$,
which is the unique real analytic function
on $\mathbb{R}\backslash\{\omega_H(\lambda)=0\}
=\mathbb{R}\backslash\{\lambda_1, \lambda_2, \dots, \lambda_{k_H}\}$
satisfying
\begin{equation}\label{03eqn:def of psi(lambda) for H}
\psi_H(\lambda)=\frac{\lambda}{\omega_H(\lambda)}\,.
\end{equation}
Let $\lambda_*=\lambda_1$ be the largest zero of $\omega_H(\lambda)$
and define $\psi_{H*}=\psi_H\!\restriction_{(\lambda_*,+\infty)}$.
Then $\psi_{H*}(\lambda)$ is an increasing 
function on $(\lambda_*,\infty)$ varying from $-\infty$ to $+\infty$.

\begin{proposition}\label{04prop:inverse of psi_H}
The inverse function $\psi_{H*}^{-1}(\lambda^\prime)$ 
is an increasing function on $\mathbb{R}$
and satisfies 
\[
\lim_{\lambda^\prime\rightarrow-\infty}
\psi_{H*}^{-1}(\lambda^\prime)=\lambda_*,
\qquad
\lim_{\lambda^\prime\rightarrow+\infty}
\psi_{H*}^{-1}(\lambda^\prime)=+\infty,
\qquad
\psi_{H*}^{-1}(0)=-\alpha_1-2,
\]
where $\lambda_*$ is the largest zero of $\omega_H(\lambda)$
and $\alpha_1=\min \mathrm{ev}_{\mathrm{M}}(A_H)\cup\{-2\}$.
Moreover,
\begin{equation}\label{04eqn:inverse of psi_H}
\psi_{H*}^{-1}(\lambda^\prime)
=\max\{\lambda\in\mathbb{R}\,;\, \psi_H(\lambda)=\lambda^\prime\}.
\end{equation}
\end{proposition}

\begin{proof}
Immediate from 
Proposition \ref{03prop:behaviour of psi inverse}.
\end{proof}

\subsection{Calculating $\mathrm{QEC}(G\odot H)$}

We come back to the main problem to calculate
$\mathrm{QEC}(G\odot H)=\max\lambda(\mathcal{S})$,
see Propositions \ref{02prop:starting expression for QEC} and 
\ref{03prop:reformulated basic equations}.
Using the $\omega$-function, 
we divide $\lambda(\mathcal{S})$ into three parts:
\begin{align}
\Gamma_1
&=\lambda(\mathcal{S})
  \cap\{\det(A_H+2+\lambda)\neq0\}
  \cap\{\omega_H(\lambda)\neq0\},
\label{03eqn:def of Gamma1}\\
\Gamma_2
&=\lambda(\mathcal{S})
  \cap\{\det(A_H+2+\lambda)\neq0\}
  \cap\{\omega_H(\lambda)=0\},
\label{03eqn:def of Gamma2}\\
\Gamma_3
&=\lambda(\mathcal{S})
  \cap\{\det(A_H+2+\lambda)=0\}.
\label{03eqn:def of Gamma3}
\end{align}
Since $\lambda(\mathcal{S})=\Gamma_1\cup \Gamma_2\cup \Gamma_3$
(disjoint union), we have
\begin{equation}\label{04eqn:our goal divided}
\mathrm{QEC}(G\odot H)
=\max\,\Gamma_1\cup \Gamma_2\cup\Gamma_3\,.
\end{equation}
This subsection is devoted to 
$\max\,\Gamma_1\cup \Gamma_2$.

\begin{lemma}\label{03lem: basic equations under det neq 0}
Assume that $\det(A_H+2+\lambda)\neq0$.
Then equations 
\eqref{02eqn:basic equation (11)}--\eqref{02eqn:basic equation (14)}
are equivalent to the following ones:
\begin{align}
&\omega_H(\lambda)D_G\xi_o-(J_G+\lambda)\xi_o 
=\frac{\mu}{2}\,\mathbf{1},
\label{03eqn:main equation (31)} \\
&\xi_x
=\lambda \langle e_x, (A_H+2+\lambda)^{-1}\mathbf{1}\rangle\xi_o,
\qquad x\in V_2,
\label{03eqn:main equation (32)} \\
&\omega_H(\lambda)J_G\xi_o=0,
\label{03eqn:main equation (33)} \\
&\langle \xi_o, \xi_o\rangle
+\sum_{x\in V_2} \langle \xi_x, \xi_x\rangle=1.
\label{03eqn:main equation (34)}
\end{align}
Therefore, $\Gamma_1\cup \Gamma_2$ coincides with
the set of $\lambda\in\mathbb{R}$ appearing 
as a solution $(\xi_o,\{\xi_x\,;\, x \in V_2\}, \lambda, \mu)$
to equations
\eqref{03eqn:main equation (31)}--\eqref{03eqn:main equation (34)}.
\end{lemma}

\begin{proof}
Since the matrix $A_H+2+\lambda$ is invertible by assumption,
from \eqref{02eqn:basic equation (12)} we obtain
\begin{align*}
\xi_x
&=\lambda \sum_{y\in V_2} 
 \langle e_x, (A_H+2+\lambda)^{-1}e_y\rangle\xi_o
\\
&=\lambda \langle e_x, (A_H+2+\lambda)^{-1}\mathbf{1}\rangle\xi_o\,,
\qquad x\in V_2,
\end{align*}
which shows \eqref{03eqn:main equation (32)}.
Then applying \eqref{03eqn:main equation (32)} to
\eqref{02eqn:basic equation (11)}, we obtain
\begin{align*}
&(D_G-J_G-\lambda)\xi_o 
 +\sum_{x\in V_2}
 \lambda \langle e_x, (A_H+2+\lambda)^{-1}\mathbf{1}\rangle D_G\xi_o
\\
&\qquad =(D_G-J_G-\lambda)\xi_o 
 +\lambda \langle\mathbf{1},
 (A_H+2+\lambda)^{-1}\mathbf{1}\rangle D_G\xi_o
=\frac{\mu}{2}\,\mathbf{1},
\end{align*}
which is equivalent to \eqref{03eqn:main equation (31)}.
Again by using \eqref{03eqn:main equation (32)},
we see that \eqref{03eqn:main equation (3)} is brought into
\begin{align*}
-J_G\xi_o
=\sum_{x\in V_2} J_G\xi_x
&=\sum_{x\in V_2}  \lambda 
 \langle e_x, (A_H+2+\lambda)^{-1}\mathbf{1}\rangle J_G\xi_o
\\
&=\lambda 
  \langle \mathbf{1}, (A_H+2+\lambda)^{-1}\mathbf{1}\rangle J_G\xi_o,
\end{align*}
namely,
\[
(1+\lambda \langle \mathbf{1}, (A_H+2+\lambda)^{-1}\mathbf{1}\rangle)
J_G\xi_o=0,
\]
which is equivalent to \eqref{03eqn:main equation (33)}.
Thus, equations
\eqref{02eqn:basic equation (11)}--\eqref{02eqn:basic equation (14)}
are equivalently brought into 
\eqref{03eqn:main equation (31)}--\eqref{03eqn:main equation (34)}.
\end{proof}

\begin{lemma}\label{03lem:description Gamma2}
Let $\lambda(\mathcal{S}_G)$ be the set of
$\lambda^\prime\in\mathbb{R}$ appearing as a 
solution $(\xi^\prime,\lambda^\prime,\mu^\prime)
\in C(V_1)\times \mathbb{R}\times \mathbb{R}$
to the following equations:
\begin{equation}\label{03eqn:for S(G)}
(D_G-\lambda^\prime I)\xi^\prime=\frac{\mu^\prime}{2}\mathbf{1},
\qquad
\langle \mathbf{1},\xi^\prime\rangle=0,
\qquad
\langle \xi^\prime,\xi^\prime\rangle=1.
\end{equation}
It then holds that
\begin{equation}\label{03eqn:description Gamma2}
\Gamma_1
=\{\det(A_H+2+\lambda)\neq0\}
 \cap\{\lambda\in\mathbb{R}\,;\,
  \psi_H(\lambda)\in \lambda(\mathcal{S}_G)\}.
\end{equation}
\end{lemma}

\begin{proof}
In view of the definition \eqref{03eqn:def of Gamma1},
for \eqref{03eqn:description Gamma2} it is sufficient to show that
\[
\lambda\in \lambda(\mathcal{S})
\quad\Leftrightarrow\quad
\psi_H(\lambda)
=\frac{\lambda}{\omega_H(\lambda)}\in \lambda(\mathcal{S}_G)
\]
under the conditions
$\det(A_H+2+\lambda)\neq0$ and $\omega_H(\lambda)\neq0$.

($\Rightarrow$) Assume first that 
$\det(A_H+2+\lambda)\neq0$, $\omega_H(\lambda)\neq0$
and $\lambda\in \lambda(\mathcal{S})$,
and take a solution $(\xi_o,\{\xi_x\,;\,\xi\in V_1\},\lambda,\mu)$
to equations
\eqref{03eqn:main equation (31)}--\eqref{03eqn:main equation (34)}.
Since $\omega_H(\lambda)\neq 0$,
\eqref{03eqn:main equation (31)} becomes
\begin{equation}\label{03eqn:main equation (51)}
D_G\xi_o-\frac{\lambda}{\omega_H(\lambda)} \xi_o
=\frac{1}{\omega_H(\lambda)}\,\frac{\mu}{2}\mathbf{1},
\end{equation}
and \eqref{03eqn:main equation (33)} becomes
\begin{equation}\label{03eqn:basic equation (53)}
\langle \mathbf{1},\xi_o \rangle=0.
\end{equation}
Moreover, applying \eqref{03eqn:main equation (32)}
to \eqref{03eqn:main equation (34)}, we obtain
\begin{equation}\label{03eqn:basic equation (54)(1)}
\big(1+\lambda^2\|(A_H+2+\lambda)^{-1}\mathbf{1}\|^2\big)
\langle \xi_o, \xi_o\rangle=1.
\end{equation}
We define $(\xi^\prime,\lambda^\prime,\mu^\prime)\in
C(V_1)\times \mathbb{R}\times\mathbb{R}$ by
\begin{equation}\label{03eqn:in proof Lemma 3.5 (1)}
\xi^\prime
=\big(1+\lambda^2\|(A_H+2+\lambda)^{-1}\mathbf{1}\|^2\big)^{1/2}
\xi_o\,,
\quad
\lambda^\prime=\frac{\lambda}{\omega_H(\lambda)},
\quad
\mu^\prime=\frac{\mu}{\omega_H(\lambda)}.
\end{equation}
Then $(\xi^\prime,\lambda^\prime,\mu^\prime)$ becomes 
a solution to \eqref{03eqn:for S(G)}
and hence $\lambda^\prime=\psi_H(\lambda)\in \lambda(\mathcal{S}_G)$.

($\Leftarrow$) Assume that 
$\det(A_H+2+\lambda)\neq0$, $\omega_H(\lambda)\neq0$ 
and $\psi_H(\lambda)\in \lambda(\mathcal{S}_G)$.
We set $\lambda^\prime=\psi_H(\lambda)$ and
take a solution $(\xi^\prime,\lambda^\prime,\mu^\prime)$ to
equations \eqref{03eqn:for S(G)}.
Then, defining $\xi_o$ and $\mu$ by 
\eqref{03eqn:in proof Lemma 3.5 (1)}
and $\xi_x$ by \eqref{03eqn:main equation (32)},
we see that $(\xi_o,\{\xi_x\,;\, x\in V_2\}, \lambda,\mu)$ becomes a
solution to equations 
\eqref{03eqn:main equation (31)}--\eqref{03eqn:main equation (34)},
and hence $\lambda\in \lambda(\mathcal{S})$.
\end{proof}

\begin{lemma}\label{04lem:description Gamma1}
If $-2-\psi_{H*}^{-1}(\mathrm{QEC}(G)) \not\in\mathrm{ev}(A_H)$,
then $\Gamma_1\neq\emptyset$ and
\begin{equation}\label{03eqn:description max Gamma2}
\max \Gamma_1=\psi_{H*}^{-1}(\mathrm{QEC}(G)).
\end{equation}
\end{lemma}

\begin{proof}
It is known by the general theory that 
\begin{equation}\label{04eqn:in proof Prop 4.5 (0)}
\mathrm{QEC}(G)=\max\lambda(\mathcal{S}_G).
\end{equation}
In particular, $\mathrm{QEC}(G)\in \lambda(\mathcal{S}_G)$
and hence
\[
\psi_{H*}^{-1}(\mathrm{QEC}(G))\in
\{\lambda\in\mathbb{R}\,;\,
  \psi_H(\lambda)\in \lambda(\mathcal{S}_G)\}.
\]
Then, together with the assumption that
$-2-\psi_{H*}^{-1}(\mathrm{QEC}(G)) \not\in\mathrm{ev}(A_H)$,
we see by Lemma \ref{03lem:description Gamma2}
that $\psi_{H*}^{-1}(\mathrm{QEC}(G))\in \Gamma_1$.
Therefore, $\Gamma_1\neq\emptyset$ and
\begin{equation}\label{03eqn:in proof Prop3.6 (01)}
\max\Gamma_1
\ge\psi_{H*}^{-1}(\mathrm{QEC}(G)).
\end{equation}
On the other hand,
it follows from Lemma \ref{03lem:description Gamma2}
that
\begin{align}
\max\Gamma_1
&\le \max\{\lambda\in\mathbb{R}\,;\,
  \psi_H(\lambda)\in \lambda(\mathcal{S}_G)\}
\nonumber \\
&=\max \{\psi_{H*}^{-1}(\lambda^\prime)\,;\,
\lambda^\prime\in \lambda(\mathcal{S}_G)\},
\label{04eqn:in proof Prop 4.5 (1)}
\end{align}
where the last equality is due to Proposition
\ref{04prop:inverse of psi_H}.
Since $\psi_{H*}^{-1}(\lambda^\prime)$ is an increasing function,
in view of \eqref{04eqn:in proof Prop 4.5 (0)} we obtain
\begin{equation}\label{04eqn:in proof Prop 4.5 (2)}
\max \{\psi_{H*}^{-1}(\lambda^\prime)\,;\,
\lambda^\prime\in \lambda(\mathcal{S}_G)\}
=\psi_{H*}^{-1}(\mathrm{QEC}(G)).
\end{equation}
Combining \eqref{04eqn:in proof Prop 4.5 (1)}
and \eqref{04eqn:in proof Prop 4.5 (2)},
we come to
\begin{equation}\label{03eqn:in proof Prop3.6 (1)}
\max\Gamma_1
\le\psi_{H*}^{-1}(\mathrm{QEC}(G)).
\end{equation}
Thus, \eqref{03eqn:description max Gamma2} follows from
\eqref{03eqn:in proof Prop3.6 (01)} and
\eqref{03eqn:in proof Prop3.6 (1)}.
\end{proof}

\begin{lemma}\label{04lem:description Gamma1 and Gamma2}
If $-2-\psi_{H*}^{-1}(\mathrm{QEC}(G)) \not\in\mathrm{ev}(A_H)$,
then we have
\begin{equation}\label{04eqn:description max Gamma and Gamm2}
\max \Gamma_1\cup\Gamma_2=\psi_{H*}^{-1}(\mathrm{QEC}(G)).
\end{equation}
Moreover,
\begin{equation}\label{03eqn:estimate max Gamma2}
\lambda_*<\psi_{H*}^{-1}(-1) \le\max \Gamma_1\cup\Gamma_2,
\end{equation}
where $\lambda_*$ is the largest zero of $\omega_H(\lambda)$.
\end{lemma}

\begin{proof}
We first note that $\Gamma_2=\emptyset$ or $\max\Gamma_2\le\lambda_*$.
In fact, if $\Gamma_2 \neq\emptyset$,
by the definition of $\Gamma_2$ we have
\begin{equation}\label{04eqn:in proof Lem 4.6(1)}
\max\Gamma_2
\le\max\{\lambda\in\mathbb{R}\,;\,\omega_H(\lambda)=0\}=\lambda_*.
\end{equation}
On the other hand, 
recall that $\psi_{H*}^{-1}(\lambda^\prime)$ is
an increasing function on $\mathbb{R}$ taking
values in $(\lambda_*,+\infty)$.
Since $\mathrm{QEC}(G)\ge-1$ for any 
connected graph $G$ on two or more vertices,
with the help of Lemma \ref{04lem:description Gamma1} we have
\begin{equation}\label{04eqn:in proof Lem 4.6(2)}
\lambda_*<\psi_{H*}^{-1}(-1) 
\le \psi_{H*}^{-1}(\mathrm{QEC}(G))
=\max\Gamma_1\,.
\end{equation}
Thus, we see from
\eqref{04eqn:in proof Lem 4.6(1)} and
\eqref{04eqn:in proof Lem 4.6(2)} 
that $\max\Gamma_1\cup\Gamma_2=\max\Gamma_1$
and \eqref{04eqn:description max Gamma and Gamm2} follows.
Finally, \eqref{03eqn:estimate max Gamma2} is already clear
from \eqref{04eqn:in proof Lem 4.6(2)}.
\end{proof}

We are in a good position to 
present a provisional formula for $\mathrm{QEC}(G\odot H)$
incorporating Lemma \ref{04lem:description Gamma1 and Gamma2}.

\begin{proposition}\label{04prop:preliminary formula with gamma3}
Let $G=(V_1,E_1)$ be a connected graph with $|V_1|\ge2$
and $H=(V_2,E_2)$ an arbitrary graph with $|V_2|\ge1$.
Let $A_H$ be the adjacency matrix of $H$,
and $\omega_H(\lambda)$ and $\psi_{H}(\lambda)$
the associated functions.
If 
\begin{equation}\label{03eqn:condition by not in ev(AH)}
-2-\psi_{H*}^{-1}(\mathrm{QEC}(G)) \not\in\mathrm{ev}(A_H),
\end{equation}
then we have
\begin{equation}\label{04eqn:preliminary formula for corona graph (0)}
\mathrm{QEC}(G\odot H)
=\max \{\psi_{H*}^{-1}(\mathrm{QEC}(G)),\,\,\gamma_3\},
\end{equation}
where 
\begin{equation}\label{04eqn:def of gamma3}
\gamma_3=\max \Gamma_3
=\max\lambda(\mathcal{S})\cap\{\det(A_H+2+\lambda)=0\},
\end{equation}
understanding tacitly that
$\gamma_3=-\infty$ if $\Gamma_3=\emptyset$.
\end{proposition}

The proof is obvious by
\eqref{04eqn:our goal divided} and
Lemma \ref{04lem:description Gamma1 and Gamma2}.
For estimate of $\gamma_3$ we note the obvious inequality:
\begin{equation}\label{04eqn:max Gamma3}
\gamma_3
\le \max\{-\alpha-2\,;\,\alpha\in\mathrm{ev}(A_H)\}
=-\min \mathrm{ev}(A_H)-2.
\end{equation}
%

\subsection{Main formula for $\mathrm{QEC}(G\odot H)$}
\label{04subsec:Formula for QEC of corona graph}

Upon applying Proposition
\ref{04prop:preliminary formula with gamma3},
we need to examine $\gamma_3=\max\Gamma_3$.
But, in some special cases it may be avoided
and we may claim that
\begin{equation}\label{04eqn:preliminary formula for corona graph (00)}
\mathrm{QEC}(G\odot H)
=\psi_{H*}^{-1}(\mathrm{QEC}(G)).
\end{equation}
During this subsection, for the corona graph $G\odot H$
we always assume that 
$G=(V_1,E_1)$ is a connected graph with $|V_1|\ge2$
and $H=(V_2,E_2)$ an arbitrary graph with $|V_2|\ge1$.
The adjacency matrix of $H$ is denoted by $A_H$,
and $\omega_H(\lambda)$ and $\psi_{H}(\lambda)$ are
the associated functions, as usual.

\begin{proposition}\label{04prop:preliminary formula for corona graph}
Formula \eqref{04eqn:preliminary formula for corona graph (00)} holds if
the following two conditions are satisfied:
\begin{enumerate}
\item[\upshape (i)] $-2-\psi_{H*}^{-1}(\mathrm{QEC}(G)) \not\in\mathrm{ev}(A_H)$;
\item[\upshape (ii)] there exists $q_0\in\mathbb{R}$ such that
$\mathrm{QEC}(G\odot H)\ge q_0$ and
$\Gamma_3\cap[q_0, +\infty)=\emptyset$.
\end{enumerate}
\end{proposition}

\begin{proof}
We have $\mathrm{QEC}(G\odot H)
=\max\,(\Gamma_1\cup \Gamma_2\cup\Gamma_3)\cap[q_0,+\infty)$
by \eqref{02eqn:QEC by lower bound q_0} and
\eqref{04eqn:our goal divided}.
Then \eqref{04eqn:preliminary formula for corona graph (00)} 
follows from Lemma \ref{04lem:description Gamma1 and Gamma2}.
\end{proof}

\begin{theorem}\label{04thm:min formula (01)}
Formula \eqref{04eqn:preliminary formula for corona graph (00)} holds if
the following three conditions are satisfied:
\begin{enumerate}
\item[\upshape (i)] 
$-2-\psi_{H*}^{-1}(\mathrm{QEC}(G)) \not\in\mathrm{ev}(A_H)$;
\item[\upshape (ii)]
$\min \mathrm{ev}(A_H)\geq -2$;
\item[\upshape (iii)] $\mathrm{QEC}(G)\ge0$.
\end{enumerate}
\end{theorem}

\begin{proof}
Recall first that 
$\psi_{H*}^{-1}(0)=-\alpha_1-2$, where
$\alpha_1=\min\mathrm{ev}_M(A_H)\cup\{-2\}$,
see Proposition \ref{04prop:inverse of psi_H}.
Then, by condition (ii) we see that $\alpha_1=-2$ and
$\psi_{H*}^{-1}(0)=0$.
Since $\psi_{H*}^{-1}(\lambda^\prime)$ is an increasing function,
condition (iii) ensures that
\[
\psi_{H*}^{-1}(\mathrm{QEC}(G))\ge0.
\]
On the other hand, in view of \eqref{04eqn:max Gamma3}
and condition (ii) we have
\[
\gamma_3
\le -\min \mathrm{ev}(A_H)-2
\le 0.
\]
It then follows from Proposition 
\ref{04prop:preliminary formula with gamma3} that
\[
\mathrm{QEC}(G\odot H)
=\max\{\psi_{H*}^{-1}(\mathrm{QEC}(G)),\, \gamma_3\}
=\psi_{H*}^{-1}(\mathrm{QEC}(G)),
\]
as desired.
\end{proof}

During the above proof condition (iii) is crucial.
The following result stands in contrast to 
Theorem \ref{04thm:min formula (01)} in this aspect.

\begin{theorem}\label{04thm:min formula (02)}
Formula \eqref{04eqn:preliminary formula for corona graph (00)} holds if
the following two conditions are satisfied:
\begin{enumerate}
\item[\upshape (i)] 
$-2-\psi_{H*}^{-1}(\mathrm{QEC}(G)) \not\in\mathrm{ev}(A_H)$;
\item[\upshape (ii)]
$\min \mathrm{ev}(A_H)> -\sqrt{2}$.
\end{enumerate}
\end{theorem}

\begin{proof}
In view of \eqref{04eqn:max Gamma3} and condition (ii) we have
\[
\gamma_3
\le -\min \mathrm{ev}(A_H)-2
< -2+\sqrt2.
\]
On the other hand, we have
$\mathrm{QEC}(G\odot H)\ge -2+\sqrt2$ by Lemma \ref{01lem:lower bound}.
Setting $q_0=-2+\sqrt2$, we see that
$\mathrm{QEC}(G\odot H)\ge q_0$ and
$\Gamma_3\cap[q_0,+\infty)=\emptyset$.
Then the assertion follows from
Proposition \ref{04prop:preliminary formula for corona graph}.
\end{proof}

\begin{theorem}\label{04thm:min formula (03)}
Formula \eqref{04eqn:preliminary formula for corona graph (00)} holds if
\begin{equation}\label{04eqn:condition for min ev(AH)}
-2-\psi_{H*}^{-1}(\mathrm{QEC}(G))
<\min\mathrm{ev}(A_H).
\end{equation}
\end{theorem}

\begin{proof}
Upon applying Proposition 
\ref{04prop:preliminary formula with gamma3},
we first note that condition
\eqref{03eqn:condition by not in ev(AH)} is fulfilled
by \eqref{04eqn:condition for min ev(AH)}.
Moreover, in view of \eqref{04eqn:max Gamma3} 
and \eqref{04eqn:condition for min ev(AH)} we have
\[
\gamma_3 
\le-\min\mathrm{ev}(A_H)-2
<\psi_{H*}^{-1}(\mathrm{QEC}(G)).
\]
It then follows from Proposition 
\ref{04prop:preliminary formula with gamma3} that
\[
\mathrm{QEC}(G\odot H)
=\max \{\psi_{H*}^{-1}(\mathrm{QEC}(G)),\,\,\gamma_3\}
=\psi_{H*}^{-1}(\mathrm{QEC}(G)),
\]
as desired.
\end{proof}

\begin{remark}\normalfont
We note that a graph $H$ satisfies $\min \mathrm{ev}(A_H)>-\sqrt{2}$
if and only if it is a disjoint union of complete graphs.
In fact, let $H$ be a graph on $n\ge1$ vertices
with $\mathrm{ev}(A_H)=\{\lambda_1\ge \lambda_2\ge
\dots\ge \lambda_n\}$.
Assume that $H$ contains $P_3$ as an induced subgraph.
Then, using $\mathrm{ev}(P_3)=\{0,\pm\sqrt2\}$
and the Cauchy interlacing theorem,
we see that $\lambda_3\ge -\sqrt2 \ge \lambda_n=\min\mathrm{ev}(A_H)$.
Therefore, if a graph $H$ satisfies
$\min \mathrm{ev}(A_H)>-\sqrt{2}$,
it does not contain $P_3$ as an induced subgraph,
and hence it is a disjoint union of complete graphs.
Conversely, if $H$ is a disjoint union of complete graphs,
we have $\min\mathrm{ev}(A_H)=0$ or $=-1$.
\end{remark}

\subsection{Corona graphs $G\odot H$ with $H$ being regular}
\label{04subsec:Formula for QEC of corona graph with regular H}

\begin{lemma}\label{04lem:main ev for regular graphs}
Let $\kappa\ge0$.
If $H$ is a $\kappa$-regular graph,
then $\kappa$ is a unique main eigenvalue, that is,
$\mathrm{ev}_M(H)=\{\kappa\}$.
\end{lemma}

The proof is obvious by definition.
Moreover, it is noteworthy that
a graph has exactly one main eigenvalue if and only if it 
is regular, see e.g., \cite{Cvetkovic1978,Cvetkovic-Rowlinson-Simic2010}.

For a $\kappa$-regular graph $H$ on $n$ vertices
with $n\ge1$ and $\kappa\ge0$,
the $\omega$- and $\psi$-functions are given by
\begin{align}
\omega_H(\lambda)
&=1+\frac{n\lambda}{\kappa+2+\lambda}
=\frac{(n+1)\lambda+\kappa+2}{\kappa+2+\lambda}\,,
\quad
\lambda_*=-\frac{\kappa+2}{n+1}\,,
\label{04eqn:lambda* for regular graphs}
\\[6pt]
\psi_H(\lambda)
&=\frac{\lambda}{\omega(\lambda)}
=\frac{\lambda(\lambda+\kappa+2)}{(n+1)\lambda+\kappa+2}\,.
\label{04eqn:psi for regular graphs}
\end{align}
And the inverse function is given explicitly as follows:
\begin{align}
\psi_{H*}^{-1}(\lambda^\prime)
&=\frac{1}{2}\Big\{(n+1)\lambda^\prime-(\kappa+2)
\nonumber\\
&\qquad\qquad +\sqrt{
 ((n+1)\lambda^\prime-(\kappa+2))^2+4(\kappa+2)\lambda^\prime}\, \Big\}.
\label{03eqn:psi_H* inverse for regular graphs}
\end{align}

\begin{theorem}\label{04thm:main formula for regular graph}
Let $G=(V_1,E_1)$ be a connected graph with $|V_1|\ge2$
and $H=(V_2,E_2)$ a $\kappa$-regular graph on $n=|V_2|$
vertices, where $n\ge1$ and $\kappa\ge0$.
Then it holds that
\begin{equation}\label{03eqn:QEC(GoH) for a regular H}
\mathrm{QEC}(G\odot H)
=\psi_{H*}^{-1}(\mathrm{QEC}(G)),
\end{equation}
where $\psi_{H*}^{-1}(\lambda^\prime)$ is defined as in
\eqref{03eqn:psi_H* inverse for regular graphs},
if the following two conditions are satisfied:
\begin{enumerate}
\item[\upshape (i)] 
$-2-\psi_{H*}^{-1}(\mathrm{QEC}(G)) \not\in\mathrm{ev}(A_H)$;
\item[\upshape (ii)]
$\displaystyle -2-\min\mathrm{ev}(A_H)\le -\frac{\kappa+2}{n+1}$.
\end{enumerate}
\end{theorem}

\begin{proof}
Upon applying Proposition \ref{04prop:preliminary formula with gamma3},
we see from \eqref{04eqn:max Gamma3}, 
\eqref{04eqn:lambda* for regular graphs}
and assumption (ii) that
\[
\gamma_3
\le -\min\mathrm{ev}(A_H)-2
\le -\frac{\kappa+2}{n+1}
=\lambda_*.
\]
On the other hand, 
$\lambda_*<\psi_{H*}^{-1}(\mathrm{QEC}(G))$ is obvious
by definition of $\psi_{H*}^{-1}(\lambda^\prime)$.
Then, it follows from
Proposition \ref{04prop:preliminary formula with gamma3} that
\[
\mathrm{QEC}(G\odot H)
=\max \{\psi_{H*}^{-1}(\mathrm{QEC}(G)),\,\,\gamma_3\}
=\psi_{H*}^{-1}(\mathrm{QEC}(G)),
\]
as desired.
\end{proof}

\begin{lemma}\label{04lem:QEC(K1+H) for regular H}
Let $H$ be a $\kappa$-regular graph on $n$ vertices
with $n\ge1$ and $\kappa\ge0$.
Then 
\[
\mathrm{QEC}(K_1+H)
=-2+\max\bigg\{-\min\mathrm{ev}(A_H),\, \frac{2n-\kappa}{n+1}\bigg\}.
\]
\end{lemma}

\begin{proof}
The formula for QEC of the join of two regular graphs is known
\cite{Lou-Obata-Huang2022}.
The assertion is just a special case.
\end{proof}

\begin{proposition}\label{04prop:QEC(K1+H)=0 case}
Let $H$ be a $\kappa$-regular graph on $n$ vertices
with $n\ge1$ and $\kappa\ge0$.
Then $\mathrm{QEC}(K_1+H)\ge0$
(resp. $\mathrm{QEC}(K_1+H)\le0$)
if and only if $\min \mathrm{ev}(A_H)\le-2$
(resp. $\min \mathrm{ev}(A_H)\ge-2$).
\end{proposition}

\begin{proof}
Since $(2n-\kappa)/(n+1)<2$, from
Lemma \ref{04lem:QEC(K1+H) for regular H} we see that
$\mathrm{QEC}(K_1+H)=0$ (resp. $\mathrm{QEC}(K_1+H)>0$)
if and only if $\min\mathrm{ev}(A_H)=-2$
(resp. $\min\mathrm{ev}(A_H)<-2$).
Therefore, $\mathrm{QEC}(K_1+H)<0$
if and only if $\min\mathrm{ev}(A_H)>-2$.
\end{proof}

The following assertion supplements Theorem \ref{04thm:min formula (01)}
in the case where $H$ is a regular graph.

\begin{theorem}\label{04thm:QEC for H with minev=-2}
Let $G=(V_1,E_1)$ be a connected graph with $|V_1|\ge2$
and $H=(V_2,E_2)$ a $\kappa$-regular graph on $n=|V_2|$
vertices, where $n\ge1$ and $\kappa\ge0$.
Assume that $\min \mathrm{ev}(A_H)= -2$.
Then we have
\[
\mathrm{QEC}(G\odot H)
=\begin{cases}
\psi_{H*}^{-1}(\mathrm{QEC}(G)), & \text{if $\mathrm{QEC}(G)>0$}, \\
0, & \text{otherwise},
\end{cases}
\]
where
\begin{align*}
\psi_{H*}^{-1}(\mathrm{QEC}(G))
&=\frac{1}{2}\Big\{(n+1)\mathrm{QEC}(G)-(\kappa+2)
\\
&\qquad\qquad +\sqrt{
 ((n+1)\mathrm{QEC}(G)-(\kappa+2))^2+4(\kappa+2)\mathrm{QEC}(G)}\, \Big\}.
\end{align*}
\end{theorem}

\begin{proof}
First assume that $\mathrm{QEC}(G)>0$.
Then we have $\psi_{H*}^{-1}(\mathrm{QEC}(G))>0$ and 
\[
-2-\psi_{H*}^{-1}(\mathrm{QEC}(G))<-2
=\min \mathrm{ev}(A_H).
\]
Applying Theorem \ref{04thm:min formula (03)}, we obtain
$\mathrm{QEC}(G\odot H)=\psi_{H*}^{-1}(\mathrm{QEC}(G))$,
of which the explicit expression is given by
\eqref{03eqn:psi_H* inverse for regular graphs}.
Next assume that $\mathrm{QEC}(G)\le0$.
We then have $\mathrm{QEC}(K_1+H)=0$
by Proposition \ref{04prop:QEC(K1+H)=0 case},
and hence $\mathrm{QEC}(G\odot H)=0$
by Proposition \ref{01prop:QEC=0}.
\end{proof}

\begin{remark}\normalfont
The graphs $H$, in particular regular ones, with $\min\mathrm{ev}(A_H)>-2$
or with $\min\mathrm{ev}(A_H)=-2$ have been extensively
studied, see \cite{Cvetkovic-Rowlinson-Simic2004} and references sited therein.
Moreover, strongly regular graphs $H$ with 
$\min\mathrm{ev}(A_H)\ge -2$ are completely classified,
see e.g., \cite{Brouwer-Cohen-Neumaier1989,Brouwer-Haemers2012}.
For those graphs $\mathrm{QEC}(G\odot H)$ is
explicitly written down with the help of
Theorem \ref{04thm:QEC for H with minev=-2}.
\end{remark}

\section{Examples}

\begin{example}[$G\odot \bar{K}_n$]
\label{05ex:H=empty graph}
\normalfont
Let $G$ be a connected graph on two or more vertices,
and let $H=\bar{K}_n$ an empty graph on $n\ge1$ vertices.
Note that $H$ is a $\kappa$-regular graph with $\kappa=0$,
and that $A_H=O$ and $\mathrm{ev}(A_H)=\mathrm{ev}_M(A_H)=\{0\}$.
By simple algebra we obtain
\begin{align}
\omega_H(\lambda)
&=1+\frac{n\lambda}{2+\lambda}
=\frac{(n+1)\lambda+2}{2+\lambda}\,,
\quad
\lambda_*=-\frac{2}{n+1}\,,
\nonumber \\[6pt]
\psi_H(\lambda)
&=\frac{\lambda}{\omega(\lambda)}
=\frac{\lambda(\lambda+2)}{(n+1)\lambda+2}\,.
\nonumber \\[6pt]
\psi_{H*}^{-1}(\lambda^\prime)
&=\frac{(n+1)\lambda^\prime-2
  +\sqrt{((n+1)\lambda^\prime-2)^2+8\lambda^\prime}}{2}\,.
\label{03eqn:psi_H* inverse for empty graph}
\end{align}
Since $\mathrm{QEC}(G)\ge-1$, we have
\[
\lambda_*=-\frac{2}{n+1}
<\psi_{H*}^{-1}(-1)\le \psi_{H*}^{-1}(\mathrm{QEC}(G))
\]
and hence
\[
-\psi_{H*}^{-1}(\mathrm{QEC}(G)))-2
<-\frac{2n}{n+1}.
\]
Since $\mathrm{ev}(A_H)=\{0\}$, we have
$-\psi_{H*}^{-1}(\mathrm{QEC}(G)))-2\notin\mathrm{ev}(A_H)$.
Moreover, by $\min\mathrm{ev}(A_H)=0$ we obtain
\[
-2-\min\mathrm{ev}(A_H)\le -\frac{\kappa+2}{n+1}\,.
\]
Thus, the assumptions 
in Theorem \ref{04thm:main formula for regular graph}
are fulfilled and we conclude that
\[
\mathrm{QEC}(G\odot \bar{K}_n)
=\psi_{H*}^{-1}(\mathrm{QEC}(G)),
\]
where $\psi_{H*}^{-1}(\lambda)$ is explicitly given by
\eqref{03eqn:psi_H* inverse for empty graph}.
Thus, the formula in \cite[Theorem 1.6]{Choudhury-Nandi2023}
is reproduced,
where we note that the cluster graph
(comb graph) of $G$ and the complete bipartite graph $K_{1,n}$
is nothing else but the corona graph $G\odot \bar{K}_n$.
\end{example}

\begin{example}[double star]
\label{05ex:H=double star}
\normalfont
The corona graph $K_2\odot \bar{K}_n$ is called a \textit{double star}.
As an immediate consequence of
Example \ref{05ex:H=empty graph} with $\mathrm{QEC}(K_2)=-1$,
we have
\[
\mathrm{QEC}(K_2\odot \bar{K}_n)
=\frac{-(n+3)+\sqrt{(n+3)^2-8}}{2}\,.
\]
This is a reproduction of \cite[Theorem 1.5]{Choudhury-Nandi2023}.
\end{example}

\begin{example}[bearded complete graph]
\label{05ex:bearded complete graph}
\normalfont
For $m\ge2$ the corona graph $K_m\odot K_1$ is a 
bearded complete graph $BK_{m,n}$ with $m=n$.
As an immediate consequence of
Example \ref{05ex:H=empty graph} with $\mathrm{QEC}(K_m)=-1$,
we have
\[
\mathrm{QEC}(K_m\odot K_1)
=-2+\sqrt2\,.
\]
This is a special case of \cite[Theorem 4.3]{Baskoro-Obata2021},
where $\mathrm{QEC}(BK_{n,m})=-2+\sqrt2=\mathrm{QEC}(P_4)$ is shown
for any $2\le m\le n$.
\end{example}

\begin{example}[$G\odot pK_q$]
\label{05ex:H=pKq}\normalfont
Let $G$ be a connected graph on two or more vertices.
For $p\ge1$ and $q\ge2$
we consider $H=pK_q$, that is, the disjoint union of 
$p$ copies of the complete graph $K_q$.
Note that the case of $q=1$ falls into Example
\ref{05ex:H=empty graph}.
Obviously, $H$ is a $\kappa$-regular graph on $n=pq$ vertices
with degree $\kappa=q-1$,
As is seen easily, we have
$\mathrm{ev}(A_H)=\{-1,q-1\}$ and
$\mathrm{ev}_{\mathrm{M}}(A_H)=\{q-1\}$.
Then, by simple algebra we have
\begin{align*}
\omega_H(\lambda)
&=1+\frac{n\lambda}{\kappa+2+\lambda}
=\frac{(n+1)\lambda+\kappa+2}{\kappa+2+\lambda}\,,
\qquad \lambda_*=-\frac{\kappa+2}{n+1}\,,
\\[6pt]
\psi_H(\lambda)
&=\frac{\lambda}{\omega(\lambda)}
=\frac{\lambda(\lambda+\kappa+2)}{(n+1)\lambda+\kappa+2}\,,
\\[6pt]
\psi_{H*}^{-1}(\lambda^\prime)
&=\frac{
(n+1)\lambda^\prime-(\kappa+2)
+\sqrt{((n+1)\lambda^\prime-(\kappa+2))^2
 +4(\kappa+2)\lambda^\prime}}{2}\,.
\end{align*}
We will examine two conditions 
in Theorem \ref{04thm:main formula for regular graph}.
Since $\mathrm{QEC}(G)\ge-1$,
\[
\lambda_*=-\frac{\kappa+2}{n+1}
<\psi_{H*}^{-1}(-1)\le \psi_{H*}^{-1}(\mathrm{QEC}(G)).
\]
Moreover, using $\kappa=q-1$ and $n=pq$, we have
\[
\lambda_*=-\frac{\kappa+2}{n+1}
=-\frac{q+1}{pq+1}\ge-1.
\]
Hence $\psi_{H*}^{-1}(\mathrm{QEC}(G))>-1$ and 
$-2-\psi_{H*}^{-1}(\mathrm{QEC}(G))<-1$,
which implies that
$-2-\psi_{H*}^{-1}(\mathrm{QEC}(G))\not\in\mathrm{ev}(A_H)$.
Moreover, we have 
\[
-2-\min\mathrm{ev}(A_H)=-1
\le -\frac{\kappa+2}{n+1}.
\]
Thus, two conditions in
Theorem \ref{04thm:main formula for regular graph}
are fulfilled and we conclude that
\[
\mathrm{QEC}(G\odot pK_q)
=\psi_{H*}^{-1}(\mathrm{QEC}(G)).
\]
In an explicit form we have
\begin{align*}
\mathrm{QEC}(G\odot pK_q)
&=\frac{1}{2} \Big\{(pq+1)\mathrm{QEC}(G)-(q+1)
\\
&\qquad 
+\sqrt{((pq+1)\mathrm{QEC}(G)-(q+1))^2+4(q+1)\mathrm{QEC}(G)}\,\Big\}.
\end{align*}
This formula is valid for all $p\ge1$ and $q\ge1$, 
see Example \ref{05ex:H=empty graph}.
\end{example}

\begin{example}[$G\odot K_n$]
\label{05ex:H=Kn}\normalfont
Specializing parameters in Example \ref{05ex:H=pKq} as $p=1$
and $q=n\ge1$, we obtain
\begin{align*}
\mathrm{QEC}(G\odot K_n)
&=\frac{1}{2} \Big\{(n+1)\mathrm{QEC}(G)-(n+1)
\\
&\qquad 
+\sqrt{((n+1)\mathrm{QEC}(G)-(n+1))^2+4(n+1)\mathrm{QEC}(G)}\,\Big\}.
\end{align*}
This is a reproduction of \cite[Theorem 1.6]{Choudhury-Nandi2023},
since the cluster graph (comb graph) of $G$ and $K_{n+1}$
is nothing else but the corona graph $G\odot K_n$.
\end{example}

\begin{example}[$G\odot pC_q$]
\label{05ex:H=pCq}\normalfont
Let $G$ be a connected graph on two or more vertices.
For $p\ge1$ and $q\ge3$
we consider $H=pC_q$, that is, the disjoint union of 
$p$ copies of a cycle $C_q$.
Note that $H$ is a $\kappa$-regular graph on $n$ vertices
with $\kappa=2$ and $n=pq$.
Obviously, $\mathrm{ev}_M(H)=\{2\}$ and after simple algebra we obtain
\begin{align}
\omega_H(\lambda)
&=1+\frac{n\lambda}{\kappa+2+\lambda}
 =\frac{(n+1)\lambda+4}{\lambda+4}\,,
\qquad
\lambda_*=-\frac{4}{n+1}\,,
\nonumber \\[6pt]
\psi_H(\lambda)
&=\frac{\lambda}{\omega_H(\lambda)}
=\frac{\lambda(\lambda+4)}{(n+1)\lambda+4}\,,
\nonumber \\[6pt]
\psi_{H*}^{-1}(\lambda^\prime)
&=\frac{(n+1)\lambda^\prime-4+\sqrt{
((n+1)\lambda^\prime-4)^2+16\lambda^\prime}}{2}\,.
\label{05eqn:in Ex 5.6}
\end{align}
Here we assume that $q\ge4$ is even.
Then, noting that $\min\mathrm{ev}(A_H)=-2$,
we apply Theorem \ref{04thm:QEC for H with minev=-2} to obtain
\[
\mathrm{QEC}(G\odot H)
=\begin{cases}
\psi_{H*}^{-1}(\mathrm{QEC}(G)), & \text{if $\mathrm{QEC}(G)>0$}, \\
0, & \text{otherwise}.
\end{cases}
\]
Finally, using \eqref{05eqn:in Ex 5.6}, we obtain
\begin{align*}
\mathrm{QEC}(G\odot pC_q)
&=\frac{1}{2} \Big\{(pq+1)\mathrm{QEC}(G)-4
\\
&\qquad 
+\sqrt{((pq+1)\mathrm{QEC}(G)-4)^2+16\mathrm{QEC}(G)}\,\Big\}
\end{align*}
for $\mathrm{QEC}(G)\ge0$, and
$\mathrm{QEC}(G\odot pC_q)=0$ otherwise.
For $q\ge3$ being odd, the situation become more complicated
and the discussion will appear elsewhere.
\end{example}

In all of the above examples, the corona graphs
$G\odot H$ are considered with $H$ being a regular graph.
The next example treats the case in which $H$ is not regular.

\begin{example}\normalfont
For $p_1\ge1$, $p_2\ge1$ and $1\le q_1<q_2$
we consider $H=p_1K_{q_1}\cup p_2K_{q_2}$,
which is the disjoint union of $p_1$ copies of $K_{q_1}$ and
$p_2$ copies of $K_{q_2}$.
We have
\[
\mathrm{ev}(A_H)=\{-1,\,q_1-1,\,q_2-1\},
\qquad
\mathrm{ev}_M(A_H)=\{q_1-1,\,q_2-1\},
\]
and hence $\min\mathrm{ev}(A_H)=-1$ and
\[
\mathrm{ev}_{\mathrm{M}}(A_H)\cup \{-2\}
=\{\alpha_1<\alpha_2<\alpha_3\}
=\{-2<q_1-1<q_2-1\}.
\]
Moreover, by simple algebra we obtain
\begin{align*}
\omega_H(\lambda)
&=1+\frac{p_1q_1\lambda}{\lambda+q_1+1}
 +\frac{p_2q_2\lambda}{\lambda+q_2+1}
\\[6pt]
&=\frac{(n+1)\lambda^2
  +(n+(p_1+p_2)q_1q_2+q_1+q_2+2)\lambda
  +(q_1+1)(q_2+1)}
  {(\lambda+q_1+1)(\lambda+q_2+1)},
\\[6pt]
\psi_H(\lambda)
&=\frac{\lambda}{\omega_H(\lambda)}
\\[6pt]
&=\frac{\lambda(\lambda+q_1+1)(\lambda+q_2+1)}
  {(n+1)\lambda^2
  +(n+(p_1+p_2)q_1q_2+q_1+q_2+2)\lambda
  +(q_1+1)(q_2+1)},
\end{align*}
where $n=p_1q_1+p_2q_2$.
By elementary examination of the quadratic polynomial
in the numerator of $\omega_H(\lambda)$ we see that
$-1<\lambda_*<0$,
where $\lambda_*$ is the largest zero of $\omega_H(\lambda)$.
Since $\psi_{H*}^{-1}(\lambda^\prime)>\lambda_*$
for any $\lambda^\prime\in\mathbb{R}$, 
for any connected graph $G$ on two or more vertices
we have
\[
\psi_{H*}^{-1}(\mathrm{QEC}(G))>\lambda_*>-1,
\]
and hence
\[
-2-\psi_{H*}^{-1}(\mathrm{QEC}(G))<-1
=\min\mathrm{ev}(A_H).
\]
Thus, condition \eqref{04eqn:condition for min ev(AH)} is
fulfilled in Theorem \ref{04thm:min formula (03)}
and we conclude that
\[
\mathrm{QEC}(G\odot (p_1K_{q_1}\cup p_2K_{q_2}))
=\psi_{H*}^{-1}(\mathrm{QEC}(G)),
\]
that is,
$\mathrm{QEC}(G\odot (p_1K_{q_1}\cup p_2K_{q_2}))$
coincides with the largest solution to
$\psi_H(\lambda)=\mathrm{QEC}(G)$,
which is reduced to a cubic equation.
There are cases where the QEC 
becomes an algebraic number of a quadratic field,
for example, the case where $\mathrm{QEC}(G)=1/(p_1+p_2-1)$.
\end{example}

\begin{remark}\normalfont
It follows from Example \ref{05ex:H=pKq} that
\begin{align*}
&\mathrm{QEC}(K_2\odot K_2)=-3+\sqrt6\approx -0.5505\cdots,
\\
&\mathrm{QEC}(K_2\odot 2K_1)=\frac{-5+\sqrt{17}}{2}
\approx -0.4384\cdots,
\end{align*}
both of which are greater than $\mathrm{QEC}(K_2\odot K_1)
=-2+\sqrt2\approx -0.5857\cdots$.
For any $G$ and $H$ with $|V_1|\ge2$ and $|V_2|\ge2$,
we see that $G\odot H$ contains $K_2\odot K_2$ or $K_2\odot 2K_1$ as
an isometric subgraph.
Therefore, $\mathrm{QEC}(G\odot H)>-2+\sqrt{2}$.
In other words, $\mathrm{QEC}(G\odot H)= -2+\sqrt{2}$ holds
if and only if $G\cong K_2$ and $H\cong K_1$,
that is, $G\odot H \cong P_4$.
In this connection, see also \cite{Baskoro-Obata2021},
where all connected graphs $G$ on two or more vertices
with $\mathrm{QEC}(G)<-2+\sqrt2$ are classified.
\end{remark}



\begin{thebibliography}{99}

\bibitem{Alfakih2018}
A. Y. Alfakih:
``Euclidean Distance Matrices and Their Applications
in Rigidity Theory,'' Springer, Cham, 2018. 

\bibitem{Balaji-Bapat2007}
R. Balaji and R. B. Bapat:
\textit{On Euclidean distance matrices},
Linear Algebra Appl. {\bfseries 424} (2007), 108--117.


\bibitem{Barik-Pati-Sarma2007}
S. Barik, S. Pati and B. K. Sarma:
\textit{The spectrum of the corona of two graphs},
SIAM J. Discrete Math. 21 (2007),


\bibitem{Baskoro-Obata2021}
E. T. Baskoro and N. Obata:
\textit{Determining finite connected graphs 
along the quadratic embedding constants of paths},
Electron. J. Graph Theory Appl. {\bfseries 9} (2021), 539--560.

\bibitem{Baskoro-Obata2024}
E. T. Baskoro and N. Obata:
\textit{A classification of graphs through quadratic embedding
constants and clique graph insights},
Commun. Combinatorics Optimization {\bfseries 9} (2024), 1--21.

\bibitem{Brouwer-Cohen-Neumaier1989}
A. E. Brouwer, A. M. Cohen and A. Neumaier:
``Distance-Regular Graphs,'' Springer; Berlin Heidelberg, 1989.

\bibitem{Brouwer-Haemers2012}
A. E. Brouwer and W. H. Haemers:
``Spectra of Graphs,'' Springer, 2012.

\bibitem{Choudhury-Nandi2023}
P. N. Choudhury and R. Nandi:
\textit{Quadratic embedding constants of graphs: Bounds
and distance spectra},
Linear Algebra Appl. {\bfseries 680} (2024), 108--125.

\bibitem{Cvetkovic1978}
D. M. Cvetkovi\'c:
\textit{The main part of the spectrum, 
divisors and switching of graphs},
Publ. Inst. Math. (Beograd) (N.S.) {\bfseries 23 (37)} (1978), 31--38.

\bibitem{Cvetkovic-Rowlinson-Simic2004}
D. Cvetkovi\'c, P. Rowlinson and S. Simi\'c:
``Spectral Generalizations of Line Graphs,''
London Math. Soc. Lecture Note Ser., {\bfseries 314},
Cambridge University Press, 2004.

\bibitem{Cvetkovic-Rowlinson-Simic2010}
D. M. Cvetkovi\'c, P. Rowlinson and S. Simi\'c:
``An Introduction to the Theory of Graph Spectra,''
Cambridge University Press, 2010.

\bibitem{Frucht-Harary1970}
R. Frucht and F. Harary:
\textit{On the corona of two graphs},
Aequationes Math. {\bfseries 4} (1970), 322--325.

\bibitem{Indulal-Stevanovic2015}
G. Indulal and D. Stevanovic:
\textit{The distance spectrum of corona and cluster of two graphs},
AKCE Int. J. Graphs Combin. {\bfseries 12} (2015), 186--192.

\bibitem{Irawan-Sugeng2021}
W. Irawan and K. A. Sugeng:
\textit{Quadratic embedding constants of hairy cycle graphs},
J. Phys.: Conf. Ser. {\bfseries 1722} (2021) 012046.

\bibitem{Jaklic-Modic2013}
G. Jakli\v{c} and J. Modic:
\textit{On Euclidean distance matrices of graphs},
Electron. J. Linear Algebra {\bfseries 26} (2013), 574--589.

\bibitem{Jaklic-Modic2014}
G. Jakli\v{c} and J. Modic:
\textit{Euclidean graph distance matrices of generalizations
of the star graph},
Appl. Math. Comput. {\bfseries 230} (2014), 650--663.

\bibitem{Liberti-Lavor-Maculan-Mucherino2014}
L. Liberti, G. Lavor, N. Maculan and A. Mucherino:
\textit{Euclidean distance geometry and applications},
SIAM Rev. {\bfseries 56} (2014), 3--69.

\bibitem{Lou-Obata-Huang2022}
Z. Z. Lou, N. Obata and Q. X. Huang:
\textit{Quadratic embedding constants of graph joins},
Graphs and Combinatorics {\bfseries 38} (2022), 161 (22 pages).

\bibitem{Mlotkowski2022}
W. M\l otkowski:
\textit{Quadratic embedding constants of path graphs},
Linear Algebra Appl. {\bfseries 644} (2022), 95--107.

\bibitem{Mlotkowski-Obata2020}
W. M\l otkowski and N. Obata:
\textit{On quadratic embedding constants of star product graphs},
Hokkaido Math. J. {\bfseries 49} (2020), 129--163.

\bibitem{Mlotkowski-Obata2025a}
W. M\l otkowski and N. Obata:
\textit{Quadratic embedding constants of fan graphs and graph joins},
Linear Algebra Appl. {\bfseries 709} (2025), 58--91.

\bibitem{Mlotkowski-Obata2025b}
W. M\l otkowski and N. Obata:
\textit{Partial Chebyshev polynomials and fan graphs},
Phil. Trans. R. Soc. A {\bfseries 383} (2025): 20240417. 

\bibitem{Obata2017}
N. Obata:
\textit{Quadratic embedding constants of wheel graphs},
Interdiscip. Inform. Sci. {\bfseries 23} (2017), 171--174.

\bibitem{Obata2023a}
N. Obata:
\textit{Primary non-QE graphs on six vertices},
Interdiscip. Inform. Sci. {\bfseries 29} (2023), 141--156.

\bibitem{Obata2023b}
N. Obata:
\textit{Complete multipartite graphs of non-QE class},
Electronic J. Graph Theory Appl. {\bfseries 11} (2023), 511--527.

\bibitem{Obata-Zakiyyah2018}
N. Obata and A. Y. Zakiyyah:
\textit{Distance matrices and quadratic embedding of graphs},
Electronic J. Graph Theory Appl. {\bfseries 6} (2018), 37--60.

\bibitem{Purwaningsih-Sugeng2021}
M. Purwaningsih and K. A. Sugeng:
\textit{Quadratic embedding constants of squid graph and kite graph},
J. Phys.: Conf. Ser. {\bfseries 1722} (2021) 012047.

\bibitem{Schoenberg1935}
I. J. Schoenberg:
\textit{Remarks to Maurice Fr\'echet's article
``Sur la d\'efinition axiomatique d'une classe d'espace
distanci\'es vectoriellement applicable sur l'espace de Hilbert'',}
Ann. of Math. {\bfseries 36} (1935), 724--732.

\bibitem{Schoenberg1937}
I. J. Schoenberg:
\textit{On certain metric spaces arising from Euclidean spaces
by a change of metric and their imbedding in Hilbert space},
Ann. of Math. {\bfseries 38} (1937), 787--793.

\bibitem{Schoenberg1938a}
I. J. Schoenberg:
\textit{Metric spaces and completely monotone functions},
Ann. of Math. {\bfseries 39} (1938), 811--841.

\bibitem{Schoenberg1938b}
I. J. Schoenberg:
\textit{Metric spaces and positive definite functions},
Trans. Amer. Math. Soc. {\bfseries 44} (1938), 522--536.

\end{thebibliography}
\end{document}